\pdfoutput=1
\RequirePackage{ifpdf}
\ifpdf % We~are running pdfTeX in pdf mode
\documentclass[pdftex]{sigma}
\else
\documentclass{sigma}
\fi

\usepackage{tikz}
\usepackage{scalerel}
\usepackage{mathtools}

\numberwithin{equation}{section}

\newtheorem{Theorem}{Theorem}[section]
\newtheorem{Corollary}[Theorem]{Corollary}
\newtheorem{Lemma}[Theorem]{Lemma}
 { \theoremstyle{definition}
\newtheorem{Definition}[Theorem]{Definition}
\newtheorem{Example}[Theorem]{Example}
\newtheorem{Remark}[Theorem]{Remark} }

\usepackage{mleftright}
\mleftright

\newcommand{\ot}{\otimes}
\newcommand{\co}{\circ}

\DeclareMathOperator{\barast}{\,\overline{\ast}\,}

\DeclareMathOperator{\Hom}{Hom}

\allowdisplaybreaks
\graphicspath{{DRAWINGS/}}

\begin{document}
\allowdisplaybreaks

\newcommand{\arXivNumber}{2402.16704}

\renewcommand{\PaperNumber}{024}

\FirstPageHeading

\ShortArticleName{Twisted Post-Hopf Algebras, Twisted Relative Rota--Baxter Operators and Hopf Trusses}

\ArticleName{Twisted Post-Hopf Algebras, Twisted Relative\\ Rota--Baxter Operators and Hopf Trusses}

\Author{Jos\'e Manuel FERN\'ANDEZ VILABOA~$^{\rm ab}$, Ram\'on GONZ\'ALEZ RODR\'IGUEZ~$^{\rm ac}$\newline and Brais RAMOS P\'EREZ~$^{\rm ab}$}

\AuthorNameForHeading{J.M.~Fern\'andez Vilaboa, R.~Gonz\'alez Rodr\'{\i}guez and B. Ramos P\'erez}

\Address{$^{\rm a)}$~CITMAga, 15782 Santiago de Compostela, Spain}

\Address{$^{\rm b)}$~Departamento de Matem\'aticas, Facultade de Matem\'aticas,\\
\hphantom{$^{\rm b)}$}~Universidade de Santiago de Compostela, 15771 Santiago de Compostela, Spain}
\EmailD{\href{mailto:josemanuel.fernandez@usc.es}{josemanuel.fernandez@usc.es}, \href{mailto:braisramos.perez@usc.es}{braisramos.perez@usc.es}}
\URLaddressD{\url{https://sites.google.com/view/braisramos/home}}

\Address{$^{\rm c)}$~Departamento de Matem\'{a}tica Aplicada II, Universidade de Vigo,\\
\hphantom{$^{\rm c)}$}~E.E. Telecomunicaci\'on, 36310 Vigo, Spain}
\EmailD{\href{mailto:rgon@dma.uvigo.es}{rgon@dma.uvigo.es}}
\URLaddressD{\url{https://dma.uvigo.es/~rgon/}}

\ArticleDates{Received November 19, 2024, in final form April 06, 2025; Published online April 15, 2025}

\Abstract{The present article is devoted to studying the categorical relationships between the categories of Hopf trusses, weak twisted post-Hopf algebras, introduced by Wang (2023), and weak twisted relative Rota--Baxter operators. The latter objects are a generalisation of the relative Rota--Baxter operators defined by Li--Sheng--Tang (2024), where the Rota--Baxter condition is modified through a cocycle. Under certain conditions, this work shows that the three aforementioned categories are equivalent.}


\Classification{18M05; 16T05; 17B38}

\section{Introduction}

An important issue in the field of mathematical physics consists in the study of the solutions of the quantum Yang--Baxter equation. Drinfeld has successfully tackled the challenge of constructing its solutions and proposed in \cite{DR1} to focus on the study of the set-theoretical ones. In~order to study this kind of solutions, Rump introduced in \cite{Rump} the notion of brace, which was subsequently generalized for the non-abelian setting by Guarnieri and Vendramin in \cite{GV}, who introduced the concept of skew brace.

A skew brace is a pair of groups, $(G,.)$ and $(G,\star)$, satisfying the following compatibility condition:
\begin{equation}\label{compatskbrace}
g\star(h. t)=(g\star h). g^{-1}. (g\star t)
\end{equation}
for all $g,h,t\in G$ and where $g^{-1}$ denotes the inverse of $g$ for the group structure $(G,.)$. The~importance of these objects lies in the fact that they induce non-degenerate and non necessarily involutive solutions of the quantum Yang--Baxter equation. Through a linearisation process, Angiono, Galindo and Vendramin in \cite{AGV} obtained Hopf braces, which are also relevant from a~physical point of view because the subclass of the cocommutative ones also gives rise to solutions of the above mentioned equation. Since the emergence of Hopf braces, many structures were born as a generalization of these, being Hopf trusses the most notable and the ones we are going to be interested in throughout this paper. Hopf trusses were defined by Brzezi\'nski in \cite{BRZ1} as the quantum version of skew trusses, which consist of a group structure $(A,\diamond)$ together with an associative operation $\circ\colon A\times A\rightarrow A$ and a function $\sigma\colon A\rightarrow A$ such that the relation
\begin{equation}\label{compatsktruss}
a\circ(b\diamond c)=(a\circ b)\diamond \sigma(a)^{\diamond}\diamond(a \circ c)
\end{equation}
holds for all $a,b,c\in A$ and where $\sigma(a)^{\diamond}$ denotes the inverse of $\sigma(a)$ in the group $(A,\diamond)$. It is easy to see that \eqref{compatsktruss} becomes \eqref{compatskbrace} when $\sigma$ is the identity.

Due to the fact that cocommutative Hopf braces give rise to solutions of the quantum Yang--Baxter equation, some objects appeared with the aim of characterising these structures. Examples of these are brace triples and post-Hopf algebras, whose respective categories are isomorphic to the category of cocommutative Hopf braces as can be consulted in \cite[Theorems~2.18 and~3.16]{FGR} (see also \cite{LST}). Thus, a natural question that arises at this point is what are the analogues of post-Hopf algebras and brace triples for Hopf trusses. In order to answer it, Wang defined in \cite{Wang} the notion of weak twisted post-Hopf algebra in the category of vector spaces over a field $\mathbb{F}$, and proved that the category of these objects is isomorphic to the category of cocommutative Hopf trusses. The main difference between weak twisted post-Hopf algebras and classical post-Hopf algebras is that the structure of the new ones is modified using an endomorphism $\Phi_{H}\colon H\rightarrow H$, where $H$ denotes the underlying Hopf algebra, called the cocycle. Therefore, in Section \ref{s2} of this article, we first give the notion of (weak) twisted post-Hopf algebra for an arbitrary braided monoidal category {\sf C}. After that, we prove that the categories of Hopf trusses and weak twisted post-Hopf algebras are isomorphic requiring a weaker hypothesis than cocommutativity which is related with the so-called cocommutativity class introduced in \cite{CCH} (see Theorem \ref{th-iso-wtph-htr}). To~conclude, we show that, if $\Phi_{H}$ preserves the unit, $\Phi_{H}$ is an idempotent morphism. Then if the base category {\sf C} admits split idempotents, under suitable conditions, we have a new Hopf algebra structure over $I(\Phi_{H})$, which is the image of $\Phi_{H}$ (see Theorem \ref{bialgIphi} and Corollary \ref{halgIphi}).

On the other hand, this paper also explores structures associated with Rota--Baxter operators. These objects were born in \cite{BAXRB} in the setting of differential operators on commutative Banach algebras and intensively studied from a probabilistic and combinatorial point of view after \cite{RORB}. Then, this notion was extended for cocommutative Hopf algebras by Goncharov in \cite{Goncharov}. In this work, we are going to be interested in the so-called relative Rota--Baxter operators, which are the most recent generalization of Rota--Baxter operators introduced by Li et al. in \cite{LST}. These objects are no more than coalgebra morphisms $T\colon H\rightarrow B$ between two Hopf algebras in the category of vector spaces over a field $\mathbb{F}$, such that $H$ with the action $\rightharpoonup$ is a left $B$-module bialgebra and the following equality holds:
\begin{equation}\label{compatrRB-intro}
T(a)T(b)=T\bigl(a_{(1)}\bigl(T(a_{(2)})\rightharpoonup b\bigr)\bigr).
\end{equation}
Goncharov's operators are a particular case of the previous ones taking $H=B$ and the adjoint action. The relevance of relative Rota--Baxter operators lies in the fact that there exists a~correspondence between them and Hopf braces which induces an adjunction between the functors involved (see \cite[Theorem 3.3]{LST}). So, the aim of Section \ref{s3} is going to be trying to generalize previous correspondence for Hopf trusses. A first approach to the previously stated problem was given by Li and Wang in \cite{LW2} (see also \cite{LW}). They introduced the notion of Rota--Baxter systems which are triples $(H,B_{1},B_{2})$, where $H$ is a Hopf algebra in the category of vector spaces and $B_{k}\colon H\rightarrow H$, $k=1,2$, are coalgebra morphisms verifying that
\begin{align*}%\label{compatRBSyst}
B_{k}(a)B_{k}(b)=B_{k}\bigl(B_{1}(a_{(1)})bS\bigl(B_{2}(a_{(2)})\bigr)\bigr)
\end{align*}
for $k=1,2$ and for all $a,b\in H$, where $S$ denotes the antipode for the Hopf algebra $H$. Take into account that every Rota--Baxter operator $T\colon H\rightarrow H$ in Goncharov's sense gives an example of Rota--Baxter system taking $B_{1}(a)=a_{(1)}T(a_{(2)})$ for all $a\in H$ and $B_{2}=T$. With these particular structures, Li and Wang are capable of constructing Hopf trusses from Rota--Baxter systems (see \cite[Proposition 3.8]{LW2}), but not the other way round. Hence, in order to achieve the desired correspondence with the Hopf trusses category, what we have done is to introduce the notion of (weak) twisted relative Rota--Baxter operator. These objects differ from the usual ones in the fact that \eqref{compatrRB-intro} is amended using a cocycle $\Psi_{H}\colon H\rightarrow H$. Therefore, Section~\ref{s3} is organized as follows. After defining the category of weak twisted relative Rota--Baxter operators, we first construct a functor from the subcategory of these objects satisfying condition \eqref{classcocomfrakmH} to the category of Hopf trusses (see Theorem \ref{funOmega}). Then, we prove that every Hopf truss verifying~\eqref{classcocomGH} induces a weak twisted relative Rota--Baxter (see Theorem \ref{funLambda}) and also show that these two functors give rise to an adjoint pair (see Theorem \ref{adjpair}) which supposes the sought-after generalization to this context of \cite[Theorem 3.3]{LST}. Note that conditions \eqref{classcocomGH} and \eqref{classcocomfrakmH} are again related with the cocommutativity class of a Hopf algebra introduced in \cite{CCH}. To conclude, we show that if we consider the subcategory of weak twisted relative Rota--Baxter operators satisfying \eqref{classcocomfrakmH} whose objects are isomorphisms too, then the previous adjoint pair induce a~categorical equivalence between the respective subcategories (see Theorem \ref{catequivHTr}).

\section{Preliminaries}
A monoidal category {\sf C} is a category endowed with a tensor functor $\otimes\colon {\sf C}\times{\sf C}\rightarrow {\sf C}$, a unit object~$K$ and families of natural isomorphisms in {\sf C},
\begin{gather*}
a_{M,N,P}\colon \ (M\ot N)\ot P\rightarrow M\ot (N\ot P),\qquad r_{M}\colon \ M\ot K\rightarrow M, \qquad l_{M}\colon \ K\ot M\rightarrow M,
\end{gather*}
called associativity, right and left unit constraints, respectively, verifying the pentagon axiom~\eqref{PA} and the triangle axiom~\eqref{TA}, which refer to the identities
\begin{gather}\label{PA}\tag{PA}
a_{M,N, P\ot Q}\co a_{M\ot N,P,Q}= ({\rm id}_{M}\ot a_{N,P,Q})\co a_{M,N\ot P,Q}\co (a_{M,N,P}\ot {\rm id}_{Q}),\\
\label{TA}\tag{TA}
({\rm id}_{M}\ot l_{N})\co a_{M,K,N}=r_{M}\ot {\rm id}_{N},
\end{gather} where ${\rm id}_{X}$ denotes the identity morphism for all object $X\in {\sf C}$. It is an important result that every monoidal category is monoidal equivalent to a strict one (see \cite[Proposition XI.5.1]{K}), which is the name given to a monoidal category where the constraints mentioned before are identities. Therefore, every monoidal category can be assumed to be strict without loss of generality and also every result proved in a strict setting hold in the general framework.

Given a monoidal category {\sf C}, it is said that {\sf C} is braided if there exists a braiding, which is a family of natural isomorphisms in {\sf C}
\[
c_{M,N}\colon \ M\ot N\rightarrow N\ot M
\]
subject to conditions
\[
c_{M,N\ot P}= (N\ot c_{M,P})\co (c_{M,N}\ot P),\qquad
c_{M\ot N, P}= (c_{M,P}\ot N)\co (M\ot c_{N,P}).
\] A. Joyal and R. Street in \cite{JS2} introduced these particular kind of categories as a tool for studying links and braids in topology. When the braiding $c$ also satisfies that $c_{Y,X}\circ c_{X,Y}={\rm id}_{X\otimes Y}$ for each~$X,Y\in {\sf C}$, $c$ is called a symmetry for {\sf C} and, in this situation, {\sf C} is said to be a symmetric monoidal category. Examples of this kind of categories are, between others, the category of vector spaces over a field $\mathbb{F}$, ${}_{\mathbb{F}}{\sf Vect}$, and also the category of modules over a commutative ring~$R$,~${}_{R}{\sf Mod}$, where the tensor functor is defined by the usual tensor product of modules and the braiding is the flip. For further information the reader is referred to \cite{Mac}.

Thus, throughout this paper, we will denote by ${\sf C}=({\sf C},\otimes, K,c)$ a strict braided monoidal category. Thanks to being working with a strict category, the identities $c_{X,K}={\rm id}_{X}=c_{K,X}$ hold for all $X\in {\sf C}$. Moreover, if $f\colon X\rightarrow Y$ is a morphism in {\sf C} and $Z$ an object, $Z\otimes f$ and $f\otimes Z$ will be used to write ${\rm id}_{Z}\otimes f$ and $f\otimes {\rm id}_{Z}$, respectively. In some instances, we will employ the standard graphical calculus in braided monoidal categories. Recall that, in graphical notation, the composition of morphisms is represented from the top to the bottom, the identity is described as a vertical line, and the tensor product of morphisms is expressed through horizontal concatenation. If $f\colon A\otimes B\rightarrow C$ is a morphism in the category {\sf C}, then it will be represented by
\[
\begin{tikzpicture}
		%\draw[step=0.25,color=red!50!white,very thin] (0,0) grid (16,1);
		\draw (8,0.5) circle (0.25cm);
		\node at (8,0.5) {\footnotesize{$f$}};
		\draw (8.175,0.675) .. controls (8.4,0.8) .. (8.6,1);
		\draw (7.825,0.675) .. controls (7.6,0.8) .. (7.4,1);
		\draw (8,0.25) -- (8,0);
\end{tikzpicture}
\]
and, if $g\colon A\rightarrow B\otimes C$ is a morphism in {\sf C}, then it will represented by
\[
	\begin{tikzpicture}
		% C\'{\i}rculo con la etiqueta
		\draw (8,-0.5) circle (0.25cm);
		\node at (8,-0.5) {\footnotesize{$g$}};
		% L\'{\i}nea vertical reflejada
		\draw (8,-0.25) -- (8,0);
		
		% L\'{\i}neas curvas reflejadas respecto al eje X
		\draw (8.175,-0.675) .. controls (8.4,-0.8) .. (8.6,-1);
		\draw (7.825,-0.675) .. controls (7.6,-0.8) .. (7.4,-1);
	\end{tikzpicture}
\]

Moreover, the braiding and its inverse will be represented by
\[
\begin{tikzpicture}
	%\draw[step=0.25,color=red!50!white,very thin] (0,0) grid (16,1);
	\draw (8,0.5) circle (0.25cm);
	\node at (8,0.5) {\footnotesize{$c$}};
	\draw (8.175,0.675) .. controls (8.25,0.85) and (8.4,0.9) .. (8.5,1);
	\draw (7.825,0.675) .. controls (7.75,0.85) and (7.6,0.9) .. (7.5,1);
	\draw (7.825,0.325) .. controls (7.75,0.25) and (7.6,0.2) .. (7.5,0);
	\draw (8.175,0.325) .. controls (8.25,0.25) and (8.4,0.2) .. (8.5,0);
\end{tikzpicture}
	\qquad \raisebox{4mm}{\text{and}} \qquad
\begin{tikzpicture}
	\draw (10,0.5) circle (0.25cm);
	\node at (10.02,0.54) {\footnotesize{$c^{\scaleto{-1}{3 pt}}$}};
	\draw (10.175,0.675) .. controls (10.25,0.85) and (10.4,0.9) .. (10.5,1);
	\draw (9.825,0.675) .. controls (9.75,0.85) and (9.6,0.9) .. (9.5,1);
	\draw (9.825,0.325) .. controls (9.75,0.25) and (9.6,0.2) .. (9.5,0);
	\draw (10.175,0.325) .. controls (10.25,0.25) and (10.4,0.2) .. (10.5,0);
\end{tikzpicture}
\]
respectively.

In what follows, we will sum up some basic definitions in the context in which we are working.

\begin{Definition}
A non-unital algebra in {\sf C} is a pair $A=(A,\mu_{A})$, where $A$ is an object in~{\sf C} and~$\mu_{A}\colon A\otimes A\rightarrow A$ is a morphism in {\sf C} called the product of $A$, satisfying that $\mu_{A}$ is associative, i.e., $\mu_{A}\circ (\mu_{A}\otimes A)=\mu_{A}\circ(A\otimes\mu_{A})$. A non-unital algebra $(A,\mu_{A})$ is said to be an algebra if there exists a morphism $\eta_{A}\colon K\rightarrow A$ in {\sf C}, called the unit of the algebra, such that $\mu_{A}\circ(\eta_{A}\otimes A)={\rm id}_{A}=\mu_{A}\circ(A\otimes\eta_{A})$. Following the graphical notation, for each algebra $A$, the product $\mu_{A}$ and the unit $\eta_{A}$ will be represented by
\[
\begin{tikzpicture}
	%\draw[step=0.25,color=red!50!white,very thin] (0,0) grid (16,1);
	\draw (8,0.5) circle (0.25cm);
	\node at (8,0.5) {\footnotesize{$\mu_{A}$}};
	\draw (8.175,0.675) .. controls (8.4,0.8) .. (8.6,1);
	\draw (7.825,0.675) .. controls (7.6,0.8) .. (7.4,1);
	\draw (8,0.25) -- (8,0);
\end{tikzpicture}
	\qquad \raisebox{3mm}{\text{and}} \qquad
\begin{tikzpicture}
	\node at (10,0.45) {$\eta_{A}=$};
	\draw (10.6,0.8) -- (10.6,0.1);
	\draw (10.6,0.85) circle (0.05 cm);
\end{tikzpicture}
\]

When there is no risk of confusion with the product which is being used, the circle in the previous diagram will be omitted.

Given $A=(A,\mu_{A})$ and $B=(B,\mu_{B})$ non-unital algebras, $f\colon A\rightarrow B$ is a morphism of non-unital algebras if $f\circ\mu_{A}=\mu_{B}\circ(f\otimes f)$. If $A$ and $B$ are algebras with units $\eta_{A}$ and $\eta_{B}$, respec\-tively,~$f$~is an algebra morphism if it also satisfies that $f\circ\eta_{A}=\eta_{B}$. Moreover, $A\otimes B$ admits a structure of non-unital algebra whose product is given by $\mu_{A\otimes B}\coloneqq (\mu_{A}\otimes\mu_{B})\circ(A\otimes c_{B,A}\otimes B)$ and, in the unital case, $A\otimes B$ is also an algebra with unit $\eta_{A\otimes B}\coloneqq\eta_{A}\otimes\eta_{B}$.
\end{Definition}
\begin{Definition}
A coalgebra in {\sf C} is a triple $D=(D,\varepsilon_{D},\delta_{D})$, where $D$ is an object in {\sf C} and~${\varepsilon_{D}\colon D\rightarrow K}$ and $\delta_{D}\colon D\rightarrow D\otimes D$ are morphisms in {\sf C} called the counit and the coproduct, respectively, verifying that $(\varepsilon_{D}\otimes D)\circ\delta_{D}={\rm id}_{D}=(D\otimes\varepsilon_{D})\circ\delta_{D}$ and also that $\delta_{D}$ is coassociative, which means that $(\delta_{D}\otimes D)\circ\delta_{D}=(D\otimes \delta_{D})\circ\delta_{D}$. Graphically, the coproduct $\delta_{D}$ and the counit~$\varepsilon_{D}$ will be represented by
\[
\begin{tikzpicture}
 % C\'{\i}rculo con la etiqueta
 \draw (8,-0.5) circle (0.25cm);
 \node at (8,-0.5) {\footnotesize{$\delta_{D}$}};
 % L\'{\i}nea vertical reflejada
 \draw (8,-0.25) -- (8,0);

 % L\'{\i}neas curvas reflejadas respecto al eje X
 \draw (8.175,-0.675) .. controls (8.4,-0.8) .. (8.6,-1);
 \draw (7.825,-0.675) .. controls (7.6,-0.8) .. (7.4,-1);
\end{tikzpicture}
	\qquad \raisebox{2.5mm}{\text{and}} \qquad
\begin{tikzpicture}
 \draw (10.65,-0.8) -- (10.65,-0.1); % L\'{\i}nea reflejada
 \draw (10.65,-0.85) circle (0.05 cm); % C\'{\i}rculo reflejado
 \node at (10,-0.58) {$\varepsilon_{D}=$};
\end{tikzpicture}
\]

As in the case of the product, when there is no risk of confusion with the coproduct which is being used, again the circle in the previous diagram will be omitted.

If $D=(D,\varepsilon_{D},\delta_{D})$ and $E=(E,\varepsilon_{E},\delta_{E})$ are coalgebras in {\sf C}, a morphism $f\colon D\rightarrow E$ is a~coalgebra morphism if $\delta_{E}\circ f=(f\otimes f)\circ\delta_{D}$ and $\varepsilon_{E}\circ f=\varepsilon_{D}$. Note also that the tensor product $D\otimes E$ has a natural coalgebra structure given by $\delta_{D\otimes E}\coloneqq (D\otimes c_{D,E}\otimes E)\circ(\delta_{D}\otimes\delta_{E})$ and $\varepsilon_{D\otimes E}\coloneqq\varepsilon_{D}\otimes\varepsilon_{E}$.
\end{Definition}
\begin{Definition}
Let $A=(A,\mu_{A})$ be a non-unital algebra and $M$ an object in {\sf C}. A non-unital left $A$-module is a pair $(M,\varphi_{M})$, where $\varphi_{M}\colon A\otimes M\rightarrow M$ is a morphism in {\sf C} called the action of $A$ over $M$, satisfying that $\varphi_{M}\circ(A\otimes\varphi_{M})=\varphi_{M}\circ(\mu_{A}\otimes M)$. When $A$ is an algebra with unit $\eta_{A}$ and $(M,\varphi_{M})$ is a non-unital left $A$-module, $(M,\varphi_{M})$ is said to be a left $A$-module if the condition $\varphi_{M}\circ(\eta_{A}\otimes M)={\rm id}_{M}$ holds too. Following the graphical notation, the action $\varphi_{M}$ will be represented by
\[
\begin{tikzpicture}
	%\draw[step=0.25,color=red!50!white,very thin] (0,0) grid (16,1);
	\draw (8,0.5) circle (0.25cm);
	\node at (8,0.5) {\scriptsize{$\varphi_{M}$}};
	\draw (8.175,0.675) .. controls (8.5,0.8) ..(8.6,1);
	\draw (7.825,0.675) .. controls (7.5,0.8) .. (7.4,1);
	\draw (8,0.25) -- (8,0);
\end{tikzpicture}
\]

If $(M,\varphi_{M})$ and $(N,\varphi_{N})$ are non-unital left $A$-modules, $f\colon M\rightarrow N$ is a morphism of non-unital left $A$-modules if $f$ is $A$-linear, i.e., $f\circ\varphi_{M}=\varphi_{N}\circ(A\otimes f)$. With these morphisms, non-unital left $A$-modules constitute a category for which left $A$-modules are a full subcategory.
\end{Definition}
\begin{Definition}
Consider $(X,\mu_{X})$ a non-unital algebra and $(X,\varepsilon_{X},\delta_{X})$ a coalgebra in {\sf C}. ${X=(X,\mu_{X},\varepsilon_{X},\delta_{X})}$ is a non-unital bialgebra in {\sf C} if $\mu_{X}$ is a coalgebra morphism. A non-unital bialgebra $X=(X,\mu_{X},\varepsilon_{X},\delta_{X})$ is said to be a bialgebra in {\sf C} if there exists a morphism $\eta_{X}\colon K\rightarrow X$ such that $(X,\eta_{X},\mu_{X})$ is an algebra and $\eta_{X}$ is a coalgebra morphism in {\sf C}. Note that $\eta_{X}$ and $\mu_{X}$ are coalgebra morphisms if and only if $\varepsilon_{X}$ and $\delta_{X}$ are algebra morphisms, respectively.

A morphism $f\colon X\rightarrow Y$ is a morphism of non-unital bialgebras in {\sf C} if $f$ is a morphism of non-unital algebras and coalgebras simultaneously. When $X$ and $Y$ are bialgebras, $f$ is a~bialgebra morphism if it also satisfies that $f\circ\eta_{X}=\eta_{Y}$.
\end{Definition}
Consider a left module $M$ with action $\varphi_{M}$. It is possible that $M$ may have an additional structure, e.g., it may be an algebra or a coalgebra. So, in what follows, we introduce the structures of module (co)algebra which require some compatibility conditions between the (co)algebra structure and the module action.
\begin{Definition}
Let $X=(X,\mu_{X},\varepsilon_{X},\delta_{X})$ be a non-unital bialgebra and $A=(A,\eta_{A},\mu_{A})$ an algebra in {\sf C}. A pair $(A,\varphi_{A})$ is a non-unital left $X$-module algebra if $(A,\varphi_{A})$ is a non-unital left $X$-module such that $\eta_{A}$ and $\mu_{A}$ are morphisms of non-unital left $X$-modules which means that the equalities
\begin{gather*}
%\label{mod-alg1}
\varphi_{A}\circ (X\otimes \eta_{A})=\varepsilon_{X}\otimes \eta_{A},\qquad
%\label{mod-alg2}
\varphi_{A}\circ (X\otimes \mu_{A})=\mu_{A}\circ \varphi_{A\otimes A}
\end{gather*}
hold, where $\varphi_{A\otimes A}=(\varphi_{A}\otimes \varphi_{A})\circ (X\otimes c_{X,A}\otimes A)\circ (\delta_{X}\otimes A\otimes A)$ is the left action on $A\otimes A$. When $X$ is a bialgebra, a non-unital left $X$-module algebra $(A,\varphi_{A})$ is said to be a left $X$-module algebra if $(A,\varphi_{A})$ is a left $X$-module.
\end{Definition}
\begin{Definition}
Let $X=(X,\mu_{X},\varepsilon_{X},\delta_{X})$ be a~non-unital bialgebra and $D=(D,\varepsilon_{D},\delta_{D})$ a~coalgebra in {\sf C}. A pair $(D,\varphi_{D})$ is said to be a non-unital left $X$-module coalgebra if $(D,\varphi_{D})$ is a non-unital left $X$-module satisfying that $\varepsilon_{D}$ and $\delta_{D}$ are morphisms of non-unital left $X$-modules, that is to say, the following equalities hold:
\begin{gather}\label{mod-coalg1}
\varepsilon_{D}\circ\varphi_{D}=\varepsilon_{X}\otimes\varepsilon_{D},\\\label{mod-coalg2}\delta_{D}\circ\varphi_{D}=\varphi_{D\otimes D}\circ(X\otimes\delta_{D}).
\end{gather}

Note that \eqref{mod-coalg1} and \eqref{mod-coalg2} are equivalent to the fact that $\varphi_{D}$ is a coalgebra morphism. In~case that $X$ is a bialgebra, a non-unital left $X$-module coalgebra $(D,\varphi_{D})$ is said to be a left $X$-module coalgebra if $(D,\varphi_{D})$ is a left $X$-module.
\end{Definition}
Next step consists in introducing the concept of Hopf algebra in {\sf C}. In these particular objects the convolution operation is highly relevant, so we define it first.
\begin{Definition}
Let $A=(A,\eta_{A},\mu_{A})$ be an algebra and $D=(D,\varepsilon_{D},\delta_{D})$ a coalgebra in {\sf C}. $\Hom(D,A)$ will denote the set of morphisms in {\sf C} from $D$ to $A$. With the convolution product, $f\ast g\coloneqq\mu_{A}\circ(f\otimes g)\circ\delta_{D}$, $\Hom(D,A)$ is an algebra with unit element $\eta_{A}\circ\varepsilon_{D}=\varepsilon_{D}\otimes\eta_{A}$.
\end{Definition}
\begin{Definition}
If $X=(X,\eta_{X},\mu_{X},\varepsilon_{X},\delta_{X})$ is a bialgebra in {\sf C}, we will say $X$ is a Hopf algebra in {\sf C} if there exists a morphism $\lambda_{X}\colon X\rightarrow X$, called the antipode, satisfying that
\begin{equation}\label{antipode}
\lambda_{X}\ast {\rm id}_{X}=\varepsilon_{X}\otimes\eta_{X}={\rm id}_{X}\ast\lambda_{X},
\end{equation}
that is to say, $\lambda_{X}$ is the convolution inverse of ${\rm id}_{X}$ in $\Hom(X,X)$.

A morphism between Hopf algebras, $f\colon X\rightarrow Y$, is a Hopf algebra morphism if $f$ is an algebra-coalgebra morphism. Note that both, $\lambda_{Y}\circ f$ and $f\circ\lambda_{X}$, are inverses of $f$ for the convolution in $\Hom(X,Y)$. So, due to the uniqueness of the inverse,
\begin{equation}\label{morant}
\lambda_{Y}\circ f=f\circ\lambda_{X}.
\end{equation}

Given a Hopf algebra $X$, we will say that $X$ is commutative if $\mu_{X}\circ c_{X,X}=\mu_{X}$, and cocommutative if $c_{X,X}\circ\delta_{X}=\delta_{X}$.

Moreover, in every Hopf algebra $X$, the following properties of the antipode $\lambda_{X}$ are satisfied: On the one hand, $\lambda_{X}$ is unique, antimultiplicative and anticomultiplicative, i.e., the equalities
\begin{gather}
%\label{a-antip1}
\lambda_{X}\co \mu_{X}= \mu_{X}\co (\lambda_{X}\ot \lambda_{X})\co c_{X,X},\nonumber
\\
\label{a-antip2}\delta_{X}\co \lambda_{X}=c_{X,X}\co (\lambda_{X}\ot \lambda_{X})\co \delta_{X}
\end{gather}
hold, and, on the other hand, $\lambda_{X}$ leaves the unit and the counit invariant, that is to say,
\begin{gather}
%\label{u-antip1}
\lambda_{X}\co \eta_{X}= \eta_{X},\qquad
\label{u-antip2}\varepsilon_{X}\co \lambda_{X}=\varepsilon_{X}.
\end{gather}

Thus, it is a direct consequence of the previous equalities that, when $X$ is commutative,~$\lambda_{X}$~is an algebra morphism and, in case that $X$ is cocommutative, $\lambda_{X}$ is a coalgebra morphism. In~addition, when $X$ is commutative or cocommutative, $\lambda_{X}\ast(\lambda_{X}\circ\lambda_{X})=\varepsilon_{X}\otimes\eta_{X}=(\lambda_{X}\circ\lambda_{X})\ast\lambda_{X}$ and so, owing to the uniqueness of the inverse for the convolution in $\Hom(X,X)$, $\lambda_{X}\circ\lambda_{X}={\rm id}_{X}$.

To finish, it will be relevant along the paper that if $X$ is a cocommutative Hopf algebra in {\sf C}, then the inverse of $c_{X,X}$ is $c_{X,X}$ itself, so the identity
\begin{equation}
\label{ccb}
c_{X,X}\circ c_{X,X}={\rm id}_{X\otimes X}
\end{equation}
holds (see \cite[Corollary 5]{Sch}).
\end{Definition}
\begin{Remark}
Every Hopf algebra $X$ in {\sf C} has a structure of left module algebra over itself with the so called adjoint action $\varphi_{X}^{{\rm ad}}\coloneqq \mu_{X}\circ(\mu_{X}\otimes\lambda_{X})\circ(X\otimes c_{X,X})\circ(\delta_{X}\otimes X)$. If $X$ is also cocommutative, then $\bigl(X,\varphi_{X}^{{\rm ad}}\bigr)$ is a left $X$-module algebra-coalgebra.
\end{Remark}
The proof of the following theorem can be found in \cite[Theorem 1.8]{FGR}.
\begin{Theorem}\label{th.interest}\samepage
Let $X=(X,\eta_{X},\mu_{X},\varepsilon_{X},\delta_{X},\lambda_{X})$ and $H=(H,\eta_{H},\mu_{H},\varepsilon_{H},\delta_{H},\lambda_{H})$ be Hopf algebras in $\sf{C}$ such that there exists a morphism $\varphi_{H}\colon X\otimes H\rightarrow H$ satisfying the following conditions:
\begin{itemize}\itemsep=0pt
\item[$(i)$] $\varphi_{H}\circ(X\otimes\mu_{H})=\mu_{H}\circ(\varphi_{H}\otimes\varphi_{H})\circ(X\otimes c_{X,H}\otimes H)\circ(\delta_{X}\otimes H\otimes H),$
\item[$(ii)$] $\varphi_{H}$ is a coalgebra morphism.
\end{itemize}
Then, $\varphi_{H}\circ(X\otimes\eta_{H})=\varepsilon_{X}\otimes\eta_{H}$ holds.
\end{Theorem}
As a direct corollary we obtain that, if $(H,\varphi_{H})$ is a left $X$-module coalgebra such that $\mu_{H}$ is $X$-linear, then $(H,\varphi_{H})$ is a left $X$-module algebra too.

Finally, we are going to recall the notion of Hopf truss introduced by T. Brzezi\'nski in \cite{BRZ1}, which are a generalization of Hopf braces (see \cite{AGV}). So, Hopf trusses are defined as follows in the braided monoidal setting.

\begin{Definition}\label{H-truss}
Let $H=(H,\varepsilon_{H},\delta_{H})$ be a coalgebra in $\sf{C}$. Let us assume that there are an algebra structure $\bigl(H, \eta_{H}, \mu_{H}^1\bigr)$ defined on $H$, an associative product $\mu_{H}^2\colon H\otimes H\rightarrow H$ and two endomorphisms of $H$ denoted by $\lambda_{H}$ and $\sigma_{H}$. We will say that
\[\bigl(H, \eta_{H}, \mu_{H}^{1}, \mu_{H}^{2}, \varepsilon_{H}, \delta_{H}, \lambda_{H}, \sigma_{H}\bigr)\]
is a Hopf truss in {\sf C} if
\begin{itemize}\itemsep=0pt
\item[(i)] $H_{1}=\bigl(H, \eta_{H}, \mu_{H}^{1}, \varepsilon_{H}, \delta_{H}, \lambda_{H}\bigr)$ is a Hopf algebra in {\sf C}.
\item[(ii)] $H_{2}=\bigl(H, \mu_{H}^{2}, \varepsilon_{H}, \delta_{H}\bigr)$ is a non-unital bialgebra in {\sf C}.
\item[(iii)] The morphism $\sigma_{H}$ is a coalgebra morphism satisfying the following equality:
\begin{gather*}
\mu_{H}^{2}\circ \bigl(H\ot \mu_{H}^{1}\bigr)=\mu_{H}^{1}\circ \bigl(\mu_{H}^{2}\otimes \Gamma_{H_{1}}^{\sigma_{H}} \bigr)\circ (H\otimes c_{H,H}\otimes H)\circ (\delta_{H}\otimes H\otimes H),\\
\includegraphics[scale=3]{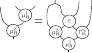}
\end{gather*}
where $\Gamma_{H_{1}}^{\sigma_{H}}\coloneqq\mu_{H}^{1}\circ \bigl((\lambda_{H}\circ\sigma_{H})\otimes \mu_{H}^{2}\bigr)\circ (\delta_{H}\otimes H),$ i.e.,
\[
\includegraphics[scale=3]{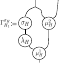}
\]
\end{itemize}

Following \cite{GONROD}, a Hopf truss will be denoted by ${\mathbb H}=(H_{1}, H_{2},\sigma_{H})$ (or simply by $\mathbb{H}$) and the morphism $\sigma_{H}$ is called the cocycle of $\mathbb{H}$.

We will say that a Hopf truss $\mathbb{H}$ is cocommutative if the underlying coalgebra $(H,\varepsilon_{H},\delta_{H})$ is cocommutative, that is to say, if $\delta_{H}=c_{H,H}\circ\delta_{H}$.
\end{Definition}
\begin{Remark}
A Hopf truss $\mathbb{H}$ is a Hopf brace if there exists a morphism $S_{H}\colon H\rightarrow H$ such that $H_{2}=\bigl(H,\eta_{H},\mu_{H}^{2},\varepsilon_{H},\delta_{H},S_{H}\bigr)$ is a Hopf algebra in $\sf{C}$ and $\sigma_{H}={\rm id}_{H}$.
\end{Remark}
\begin{Definition}
Let $\mathbb{H}=(H_{1},H_{2},\sigma_{H})$ and $\mathbb{B}=(B_{1},B_{2},\sigma_{B})$ be Hopf trusses in $\sf{C}$. A~morphism~$f$ between the two underlying objects is a morphism of Hopf trusses if $f\colon H_{1}\rightarrow B_{1}$ is a~Hopf algebra morphism and $f\colon H_{2}\rightarrow B_{2}$ is a morphism of non-unital bialgebras.

As was proved in \cite[Proposition 6.8]{BRZ1}, for every $f\colon \mathbb{H}\rightarrow \mathbb{B}$ morphism of Hopf trusses, the equality
\begin{equation}\label{Cond mor truss}
\sigma_{B}\circ f=f\circ \sigma_{H}
\end{equation}
holds.
\end{Definition}
Therefore, Hopf trusses give rise to a category that we will denote by $\sf{HTr}$. Cocommutative Hopf trusses form a full subcategory of Hopf trusses that we will denote by $\sf{coc}\textnormal{-}{\sf HTr}$.

The proofs that we can find in \cite[Section 6]{BRZ1} can be replicated in the braided monoidal setting since they do not depend on the symmetry of the category ${\mathbb F}$-{\sf Vect}. Then we have the following property: Given a Hopf truss $\mathbb{H}$ in $\sf{C}$ the second product $\mu_{H}^{2}$ admits the following expression:
\begin{equation}\label{mu2Htruss}
\mu_{H}^{2}=\mu_{H}^{1}\circ\bigl(\sigma_{H}\otimes\Gamma_{H_{1}}^{\sigma_{H}}\bigr)\circ(\delta_{H}\otimes H).
\end{equation}

By \cite[Lemma 6.2]{BRZ1}, it is known that the cocycle $\sigma_{H}$ is fully determined by $\eta_{H}$ and the product~$\mu_{H}^{2}$ in the following way:
\begin{equation*}%\label{cocycle truss}
\sigma_{H}=\mu_{H}^{2}\circ(H\otimes\eta_{H}).
\end{equation*}

Then, as a consequence of the associativity for the product $\mu_{H}^{2}$, we have that
\begin{equation}\label{cocycle truss 2}
\sigma_{H}\circ\mu_{H}^{2}=\mu_{H}^{2}\circ(H\otimes\sigma_{H})
\end{equation}
holds. In addition, by \cite[Theorem 6.5]{BRZ1} we know that $\bigl(H_{1},\Gamma_{H_{1}}^{\sigma_{H}}\bigr)$ is a non-unital left $H_{2}$-module algebra.

This introductory section concludes with the notion of finite object in $\sf{C}$.
\begin{Definition}%\label{finite}
An object $P$ in $\sf{C}$ is finite if there exists an object $P^{\ast}$, called the dual of $P$, and a $\sf{C}$-adjunction $P\otimes -\dashv P^{\ast}\otimes -$ between the tensor functors.
\end{Definition}

For example, if ${\sf C}$ is the category of left modules over a commutative ring $R$, then $P$ is finite if and only if $P$ is finitely generated and projective. Moreover, it is well known that for every finite Hopf algebra in a braided monoidal category its antipode is an isomorphism (see \cite{LP}) and, as a consequence, the associated category of Yetter--Drinfeld modules is an example of braided monoidal category.

We will denote by $a_{P}$ and $b_{P}$ the unit and the counit of the previous $\sf{C}$-adjunction, respectively. Then,
$a_{P}(K)\colon K\rightarrow P^{\ast}\otimes P$ and $b_{P}(K)\colon P\otimes P^{\ast}\rightarrow K,$
which will be represented graphically by
\[
\begin{tikzpicture}
	%\draw[step=0.25,color=red!50!white,very thin] (0,0) grid (16,1);
	\node at (6,0.5) {$a_{P}(K)=$};
	\draw (7.25,0.8) -- (7.25,0.6);
	\draw (7.25,0.85) circle (0.05 cm);
	\draw plot [domain=6.85:7.65, samples=50] (\x, {0.6 - 2*(\x-7.25)^2});
\end{tikzpicture}
	\qquad \raisebox{2.5mm}{\text{and}} \qquad
\begin{tikzpicture}
	\node at (9.55,0.5) {$b_{P}(K)=$};
	\draw (10.75,0.3) -- (10.75,0.5);
	\draw (10.75,0.25) circle (0.05 cm);
	\draw plot [domain=10.35:11.15, samples=50] (\x, {0.5 + 2*(\x-10.75)^2});
\end{tikzpicture}
\]

By the properties of the adjunction, the following equalities hold:
\begin{gather*}(b_{P}(K)\otimes P)\circ(P\otimes a_{P}(K))={\rm id}_{P},\\(P^{\ast}\otimes b_{P}(K))\circ(a_{P}(K)\otimes P^{\ast})={\rm id}_{P^{\ast}},\end{gather*}
i.e.,
\[
\includegraphics[scale=3.5]{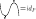} \raisebox{4.5mm}{\text{${}= \mathrm{id}_P $\qquad \text{and} \qquad}}
\includegraphics[scale=2.8]{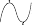} \raisebox{4.5mm}{\text{${}= \mathrm{id}_{P^*}$,}}
\]
respectively. Finite objects in $\sf{C}$ constitute a full subcategory of $\sf{C}$ that we will denote by $\sf{C^{f}}$.

Note that, for every finite object $P$ in $\sf{C}$, we have a natural algebra structure in $\sf{C}$ over the tensor object $P^{\ast}\otimes P$ as we can see in the following lemma, whose proof is straightforward.
\begin{Lemma}
Let $P$ be a finite object in $\sf{C}$, then $P^{\ast}\otimes P$ is an algebra in $\sf{C}$ with product given by $\mu_{P^{\ast}\otimes P}\coloneqq P^{\ast}\otimes b_{P}(K)\otimes P$ and unit $\eta_{P^{\ast}\otimes P}\coloneqq a_{P}(K)$.
\end{Lemma}

\section{(Weak) twisted post-Hopf algebras and Hopf trusses}\label{s2}
Post-Hopf algebras are a particular kind of structures introduced by Y. Li, Y. Sheng and R. Tang in \cite{LST} in the category of vector spaces over a field $\mathbb{F}$. These objects are specially interesting because their respective category is isomorphic to the category of Hopf braces under cocommutativity hypothesis (see \cite[Theorem 2.12]{LST}). This concept was generalized for an arbitrary braided monoidal category {\sf C} in \cite{FGR}, which is not a straightforward step, and also the categorical isomorphism with the category of Hopf braces was extended to this framework (see \cite[Theorem~3.16 and Corollary~3.17]{FGR}). Subsequently, (weak) twisted post-Hopf algebras were introduced by~S. Wang in \cite{Wang} for the category ${}_{\mathbb{F}}{\sf Vect}$ with the aim of generalising all the above-mentioned results for Hopf trusses as can be consulted along \cite[Section 6]{Wang}.

In this section, the notion of (weak) twisted post-Hopf algebra is going to be extended to the braided monoidal setting and we are going to prove two fundamental results: Firstly, an isomorphism between the categories of Hopf trusses and weak twisted post-Hopf algebras is obtained in which, in contrast with \cite[Theorem 6.7]{Wang}, the hypothesis of cocommutativity is replaced by a weaker condition and, in addition, we will also prove that, when the base category admits split idempotents, every twisted post-Hopf algebra induces a new Hopf algebra structure.
\begin{Definition}\label{WTPH}
A weak twisted post-Hopf algebra in $\sf{C}$ is a~triple $(H,m_{H},\Phi_{H})$ where $H$ is a~Hopf algebra in $\sf{C}$ and $m_{H}\colon H\otimes H\rightarrow H$ and $\Phi_{H}\colon H\rightarrow H$ are morphisms in $\sf{C}$ satisfying the following conditions:
\begin{itemize}\itemsep=0pt
	\item[(i)] $m_{H}$ is a coalgebra morphism, which means that the following equalities hold:
	\begin{itemize}\itemsep=0pt
	\item[(i.1)] $\delta_{H}\circ m_{H}=(m_{H}\otimes m_{H})\circ(H\otimes c_{H,H}\otimes H)\circ(\delta_{H}\otimes\delta_{H})$,
	\item[(i.2)] $\varepsilon_{H}\circ m_{H}=\varepsilon_{H}\otimes\varepsilon_{H}.$
	\end{itemize}
	\item[(ii)] $\Phi_{H}$ is a coalgebra morphism, that is to say
	\begin{itemize}\itemsep=0pt
	\item[(ii.1)]$\delta_{H}\circ\Phi_{H}=(\Phi_{H}\otimes\Phi_{H})\circ\delta_{H}$,
	\item[(ii.2)]$\varepsilon_{H}\circ\Phi_{H}=\varepsilon_{H}$.

	\end{itemize}
	\item[(iii)] $\Phi_{H}\circ\mu_{H}\circ(\Phi_{H}\otimes m_{H})\circ(\delta_{H}\otimes H)=\mu_{H}\circ(\Phi_{H}\otimes m_{H})\circ(\delta_{H}\otimes \Phi_{H})$, i.e.,
	\[
 \includegraphics[scale=3]{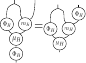}
	\]

	\item[(iv)] $m_{H}\circ(H\otimes m_{H})=m_{H}\circ((\mu_{H}\circ(\Phi_{H}\otimes m_{H})\circ(\delta_{H}\otimes H))\otimes H)$, i.e.,
	\[
	\includegraphics[scale=3]{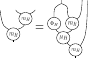}
	\]	
	\item[(v)] $m_{H}\circ(H\otimes\mu_{H})=\mu_{H}\circ(m_{H}\otimes m_{H})\circ(H\otimes c_{H,H}\otimes H)\circ(\delta_{H}\otimes H\otimes H)$, i.e.,
	\[
	\includegraphics[scale=3]{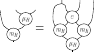}
	\]	
\end{itemize}

The morphism $\Phi_{H}$ will be called the cocycle of the weak twisted post-Hopf algebra $H$.
Moreover, if $H$ is a finite Hopf algebra and the conditions
\begin{itemize}\itemsep=0pt
	\item[(vi)] $\Phi_{H}\circ\eta_{H}=\eta_{H}$, i.e., the cocycle $\Phi_{H}$ preserves the unit,
	\item[(vii)] the morphism
	\[\alpha_{H}\coloneqq (H^{\ast}\otimes m_{H})\circ(c_{H,H^{\ast}}\otimes H)\circ(H\otimes a_{H}(K))\colon\ H\rightarrow H^{\ast}\otimes H\]
	\[
	\includegraphics[scale=3.5]{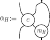}
	\]
is convolution invertible in $\Hom(H,H^{\ast}\otimes H)$, which means that there exists
 \[
 \beta_{H}\colon\ H\rightarrow H^{\ast}\otimes H
 \]
 such that
\begin{align*}
 (H^{\ast}\otimes b_{H}(K)\otimes H)\circ(\alpha_{H}\otimes\beta_{H})\circ\delta_{H}&{}=\varepsilon_{H}\otimes a_{H}(K)\\
 &{}=(H^{\ast}\otimes b_{H}(K)\otimes H)\circ (\beta_{H}\otimes\alpha_{H})\circ\delta_{H},
 \end{align*}
	i.e.,
	\[
	\includegraphics[scale=3]{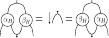}
	\]
\end{itemize}
hold, then the triple $(H,m_{H},\Phi_{H})$ is said to be a twisted post-Hopf algebra.
\end{Definition}
\begin{Definition}
Let $(H,m_{H},\Phi_{H})$ and $(B,m_{B},\Phi_{B})$ be weak twisted post-Hopf algebras in $\sf{C}$. We will say that $f\colon (H,m_{H},\Phi_{H})\rightarrow (B,m_{B},\Phi_{B})$ is a morphism of weak twisted post-Hopf algebras if $f\colon H\rightarrow B$ is a Hopf algebra morphism such that
\begin{gather}
\label{Cond1 mor TPH}f\circ m_{H}=m_{B}\circ(f\otimes f),\\
\label{Cond2 mor TPH} \Phi_{B}\circ f=f\circ \Phi_{H}.
\end{gather}
\end{Definition}
Therefore, weak twisted post-Hopf algebras give rise to a category that we will denote by ${\sf wt}\textnormal{-}\sf{Post}\textnormal{-}\sf{Hopf}$. In addition, twisted post-Hopf algebras constitute a full subcategory of weak twisted post-Hopf algebras that we will denote by ${\sf t}\textnormal{-}\sf{Post}\textnormal{-}\sf{Hopf}$. If the underlying Hopf algebra is cocommutative, the structure $(H,m_{H},\Phi_{H})$ is referred to as a cocommutative weak twisted post-Hopf algebra. The corresponding full subcategory is denoted as $\sf{coc}\textnormal{-}\sf{wt}\textnormal{-}\sf{Post}\textnormal{-}\sf{Hopf}$. In the twisted and cocommutative setting, this full subcategory is denoted as $\sf{coc}\textnormal{-}\sf{t}\textnormal{-}\sf{Post}\textnormal{-}\sf{Hopf}$.
\begin{Remark}
Note that \cite[Definition 6.3]{Wang}, that is the definition of (weak) twisted post-Hopf algebras proposed by S. Wang in the category ${}_{\mathbb{F}}\sf{Vect}$, always requires cocommutativity of the underlying Hopf algebra. In the previous definition, this requirement was omitted.

Note also that when ${\sf C}={\sf Set}$ the category of sets, which is an example of braided monoidal category whose braiding is defined by the usual flip, then $(H,m_{H},\Phi_{H})$ is a (weak) twisted post-Hopf algebra if and only $(H,m_{H},\Phi_{H})$ is a (weak) twisted post-group following the definition introduced in \cite[Definition 2.1]{Wang}.
\end{Remark}
\begin{Example}\label{ExampleTwPHA}
If ${\sf C}$ is the category of vector spaces over a field ${\mathbb F}$, then every twisted post-Hopf algebra in ${\sf C}$ is a post-Hopf algebra in the sense of Li, Sheng and Tang if $\Phi_{H}={\rm id}_{H}$. In the same setting and under finite conditions, every Yetter--Drinfeld post-Hopf algebra in Sciandra's sense (see \cite{SCI}) is an example of twisted post-Hopf algebra in a braided monoidal category (in this case this category is the category of Yetter--Drinfeld modules over the subadjacent Hopf algebra) where $\Phi_{H}={\rm id}_{H}$. Note that Yetter--Drinfeld post-Hopf algebras are usual post-Hopf algebras in the cocommutative setting and also note that the category of Yetter--Drinfeld post-Hopf algebras is isomorphic to the category of Yetter--Drinfeld braces introduced by Ferri and Sciandra in \cite{FSCI}.
\end{Example}
\begin{Lemma}
Let $(H,m_{H},\Phi_{H})$ be an object in ${\sf wt}\textnormal{-}\sf{Post}\textnormal{-}\sf{Hopf}$. The equality
\begin{equation}\label{mHprop1}
m_{H}\circ(H\otimes\eta_{H})=\varepsilon_{H}\otimes\eta_{H}
\end{equation}
holds. Moreover, when the Hopf algebra $H$ is finite,
\begin{equation}\label{lem1prop1}
m_{H}\circ c_{H,H}^{-1}=( b_{H}(K)\otimes H)\circ(H\otimes\alpha_{H})
\end{equation}
holds. Then,
\begin{equation}\label{lem1prop2}
m_{H}=( b_{H}(K)\otimes H)\circ(H\otimes\alpha_{H})\circ c_{H,H}
\end{equation}
also holds and, if $(H,m_{H},\Phi_{H})$ is an object in ${\sf t}\textnormal{-}\sf{Post}\textnormal{-}\sf{Hopf}$, we obtain that
\begin{equation}\label{mHprop2}
m_{H}\circ(\eta_{H}\otimes H)={\rm id}_{H}.
\end{equation}
\end{Lemma}
\begin{proof}
Firstly, by conditions (i) and (v) of Definition \ref{WTPH}, the proof of \eqref{mHprop1} follows by Theorem~\ref{th.interest}. Moreover, the proof of \eqref{lem1prop1} and \eqref{lem1prop2} is completely analogous to the proof of \mbox{\cite[Lemma~3.4]{FGR}}.
In addition, note that morphism $m_{H}\circ(\eta_{H}\otimes H)$ is idempotent when $(H,m_{H},\Phi_{H})$ is a~twisted post-Hopf algebra. Indeed,
\begin{gather*}
 m_{H}\circ(\eta_{H}\otimes(m_{H}\circ(\eta_{H}\otimes H)))\\
 \qquad{}= m_{H}\circ((\mu_{H}\circ(\Phi_{H}\otimes m_{H})\circ ((\delta_{H}\circ\eta_{H})\otimes\eta_{H}))\otimes H)\ \text{\footnotesize\textnormal{(by (iv) of Definition \ref{WTPH})}}\\
 \qquad{}= m_{H}\circ((\mu_{H}\circ((\Phi_{H}\circ\eta_{H})\otimes(m_{H}\circ(\eta_{H}\otimes\eta_{H}))))\otimes H)\\
 \qquad\quad{}\
 \text{\footnotesize\textnormal{(by the condition of coalgebra morphism for $\eta_{H}$)}}\\
 \qquad{}= m_{H}\circ((\mu_{H}\circ(\eta_{H}\otimes(m_{H}\circ(\eta_{H}\otimes\eta_{H}))))\otimes H)\ \text{\footnotesize\textnormal{(by (vi) of Definition \ref{WTPH})}}\\
 \qquad{}= (\varepsilon_{H}\circ\eta_{H})\otimes(m_{H}\circ(\eta_{H}\otimes H))\ \text{\footnotesize\textnormal{(by unit property and \eqref{mHprop1})}}\\
 \qquad{}= m_{H}\circ(\eta_{H}\otimes H)\ \text{\footnotesize\textnormal{(by (co)unit properties)}}.
\end{gather*}

Following the rest of the proof as in \cite[Lemma 3.5]{FGR}, \eqref{mHprop2} holds.
\end{proof}

\begin{Theorem}\label{nuHbar}
Let $(H,m_{H},\Phi_{H})$ be an object in ${\sf wt}\textnormal{-}\sf{Post}\textnormal{-}\sf{Hopf}$. If
\begin{gather}\label{classcocommH}
(m_{H}\otimes H)\circ(H\otimes c_{H,H})\circ ((c_{H,H}\circ\delta_{H})\otimes H)=(m_{H}\otimes H)\circ(H\otimes c_{H,H})\circ(\delta_{H}\otimes H),
\end{gather}
 i.e.,
\[
\includegraphics[scale=3]{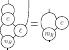}
\]
\noindent holds, then $\overline{H}=(H,\overline{\mu}_{H},\varepsilon_{H},\delta_{H})$ is a non-unital bialgebra in ${\sf C}$, where \[\overline{\mu}_{H}\coloneqq\mu_{H}\circ(\Phi_{H}\otimes m_{H})\circ(\delta_{H}\otimes H),\]
i.e.,
\[
\includegraphics[scale=3]{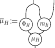}
\]

Moreover, note that
\begin{equation}\label{cond1WTPHA}\overline{\mu}_{H}\circ(H\otimes\eta_{H})=\Phi_{H},\end{equation}
and, if $(H,m_{H},\Phi_{H})$ is an object in ${\sf t}\textnormal{-}\sf{Post}\textnormal{-}\sf{Hopf}$, then
\begin{equation}\label{muetaid}
\overline{\mu}_{H}\circ(\eta_{H}\otimes H)={\rm id}_{H}.
\end{equation}
\end{Theorem}
\begin{proof}
Let us start the proof with the associative character of $\overline{\mu}_{H}$,
\begin{gather*}
 \overline{\mu}_{H}\circ (\overline{\mu}_{H}\otimes H )\\
= \mu_{H}\circ (\Phi_{H}\otimes m_{H} ) \circ (((\mu_{H}\otimes\mu_{H})\circ (H\otimes c_{H,H}\otimes H )\circ ((\delta_{H}\circ\Phi_{H})\otimes (\delta_{H}\circ m_{H})))\otimes H)\\
\quad{}\circ (\delta_{H}\otimes H\otimes H )\ \text{\footnotesize\textnormal{(by the condition of coalgebra morphism for $\mu_{H}$)}}\\
= \mu_{H}\circ (\Phi_{H}\otimes m_{H} )\circ ( ( (\mu_{H}\otimes \mu_{H} )\circ (H\otimes c_{H,H}\otimes H ) )\otimes H )\\
\quad{} \circ ( ( (\Phi_{H}\otimes \Phi_{H} )\circ\delta_{H} )\otimes ( (m_{H}\otimes m_{H} )\circ (H\otimes c_{H,H}\otimes H )\circ (\delta_{H}\otimes\delta_{H} ) )\otimes H )\\
\quad{} \circ (\delta_{H}\otimes H\otimes H )\ \text{\footnotesize\textnormal{(by (i.1) and (ii.1) of Definition \ref{WTPH})}}\\
= \mu_{H}\circ ( (\Phi_{H}\circ\mu_{H}\circ (\Phi_{H}\otimes H ) )\otimes (m_{H}\circ ( (\mu_{H}\circ (\Phi_{H}\otimes m_{H} ) )\otimes H ) ) )\\
\quad{} \circ (H\otimes ( (m_{H}\otimes H )\circ (H\otimes c_{H,H} )\circ ( (c_{H,H}\circ\delta_{H} )\otimes H ) )\otimes H\otimes H\otimes H )\\
\quad{} \circ (H\otimes ( (H\otimes c_{H,H}\otimes H )\circ (\delta_{H}\otimes\delta_{H} ) )\otimes H )\circ (\delta_{H}\otimes H\otimes H )\\
\quad{}\, \text{\footnotesize\textnormal{(by naturality of $c$ and coassociativity of $\delta_{H}$)}}\\
= \mu_{H}\\
\quad{}\circ( (\Phi_{H}\circ\mu_{H}\circ (\Phi_{H}\otimes m_{H} )\circ (\delta_{H}\otimes H ) )\otimes(m_{H}\circ( (\mu_{H}\circ (\Phi_{H}\otimes m_{H} )\circ (\delta_{H}\otimes H ) )\otimes H)))\\
\quad{}\circ ( ( (H\otimes c_{H,H}\otimes H )\circ (\delta_{H}\otimes\delta_{H} ) )\otimes H ) \ \text{\footnotesize\textnormal{(by~\eqref{classcocommH}, coassociativity of $\delta_{H}$ and naturality of $c$)}}\\
= \mu_{H}\circ ( (\mu_{H}\circ (\Phi_{H}\otimes m_{H} )\circ (\delta_{H}\otimes \Phi_{H} ) )\otimes (m_{H}\circ (H\otimes m_{H} ) ) )\\
\quad{} \circ ( ( (H\otimes c_{H,H}\otimes H )\circ (\delta_{H}\otimes\delta_{H} ) )\otimes H )\ \text{\footnotesize\textnormal{(by (iii) and (iv) of Definition \ref{WTPH})}}\\
= \mu_{H}\circ (\Phi_{H}\otimes (\mu_{H}\circ (m_{H}\otimes m_{H} )\circ (H\otimes c_{H,H}\otimes H )\circ (\delta_{H}\otimes H\otimes H ) ) )\\
\quad{} \circ (\delta_{H}\otimes ( (\Phi_{H}\otimes m_{H} )\circ (\delta_{H}\otimes H ) ) )\\
\quad{}\, \text{\footnotesize\textnormal{(by naturality of $c$, coassociativity of $\delta_{H}$ and associativity of $\mu_{H}$)}}\\
= \overline{\mu}_{H}\circ (H\otimes\overline{\mu}_{H} )\ \text{\footnotesize\textnormal{(by (v) of Definition \ref{WTPH})}}.
\end{gather*}

Moreover, $\overline{\mu}_{H}$ is a coalgebra morphism. Indeed, it is straightforward to prove that $\varepsilon_{H}\circ \overline{\mu}_{H}=\varepsilon_{H}\otimes\varepsilon_{H}$ and also
\begin{align*}
 \delta_{H}\circ\overline{\mu}_{H}={}& (\mu_{H}\otimes\mu_{H})\circ(H\otimes c_{H,H}\otimes H)\circ(\delta_{H}\otimes\delta_{H})\circ(\Phi_{H}\otimes m_{H})\circ(\delta_{H}\otimes H)\\
 &{}\text{\footnotesize\textnormal{(by the condition of coalgebra morphism for $\mu_{H}$)}}\\
 ={}& (\mu_{H}\otimes\mu_{H})\circ(H\otimes c_{H,H}\otimes H)\\
 &{}\circ(((\Phi_{H}\otimes\Phi_{H})\circ\delta_{H})\otimes((m_{H}\otimes m_{H})\circ(H\otimes c_{H,H}\otimes H)\circ(\delta_{H}\otimes\delta_{H})))\\
 &{}\circ(\delta_{H}\otimes H) \ \text{\footnotesize\textnormal{(by (i.1) and (ii.1) of Definition \ref{WTPH})}}\\
 ={}& (\mu_{H}\otimes\mu_{H})\circ(\Phi_{H}\otimes((m_{H}\otimes \Phi_{H})\circ(H\otimes c_{H,H})\circ((c_{H,H}\circ\delta_{H})\otimes H))\otimes m_{H})\\
 &{}\circ(\delta_{H}\otimes c_{H,H}\otimes H)\circ(\delta_{H}\otimes\delta_{H})\ \text{\footnotesize\textnormal{(by naturality of $c$ and coassociativity of $\delta_{H}$)}}\\
 ={}& (\mu_{H}\otimes\mu_{H})\circ(\Phi_{H}\otimes((m_{H}\otimes \Phi_{H})\circ(H\otimes c_{H,H})\circ(\delta_{H}\otimes H))\otimes m_{H})\\
 &{}\circ(\delta_{H}\otimes c_{H,H}\otimes H)\circ(\delta_{H}\otimes\delta_{H})\ \text{\footnotesize\textnormal{(by \eqref{classcocommH})}}\\
 ={}& (\overline{\mu}_{H}\otimes\overline{\mu}_{H})\circ(H\otimes c_{H,H}\otimes H)\circ(\delta_{H}\otimes\delta_{H})\\
 &{}\text{\footnotesize\textnormal{(by coassociativity of $\delta_{H}$, naturality of $c$ and definition of $\overline{\mu}_{H}$)}}.
\end{align*}

Therefore,
\[
(H,\overline{\mu}_{H},\varepsilon_{H},\delta_{H})
\]
is a non-unital bialgebra.

The proof of \eqref{cond1WTPHA} follows by (\ref{mHprop1}) and by the (co)unit properties because
\[{\mu}_{H}\circ(H\otimes\eta_{H})=\mu_{H}\circ(\Phi_{H}\otimes(\eta_{H}\circ\varepsilon_{H}))\circ\delta_{H}=\Phi_{H}.\]

To finish, let us see that \eqref{muetaid} holds when $(H,m_{H},\Phi_{H})$ is a twisted post-Hopf algebra,
\begin{gather*}
 \overline{\mu}_{H}\circ(\eta_{H}\otimes H)\\
 \qquad{}= \mu_{H}\circ((\Phi_{H}\circ\eta_{H})\otimes(m_{H}\circ(\eta_{H}\otimes H)))\ \text{\footnotesize\textnormal{(by the condition of coalgebra morphism for $\eta_{H}$)}}\\
 \qquad{}= \mu_{H}\circ(\eta_{H}\otimes H)\ \text{\footnotesize\textnormal{(by (vi) of Definition \ref{WTPH} and \eqref{mHprop2})}}\\
 \qquad{}= {\rm id}_{H}\ \text{\footnotesize\textnormal{(by unit properties)}}. \tag*{\qed}
\end{gather*}
\renewcommand{\qed}{}
\end{proof}

\begin{Remark}
If ${\sf C}$ is a symmetric category, condition \eqref{classcocommH} means that $(H,m_{H})$ is in the cocommutativity class of $\overline{H}$ following the notion introduced in \cite[Definitions~2.1 and~2.2]{CCH}. This condition has its roots in the study of the notion of bicrossed product (or double cross product) of Hopf algebras introduced by Majid in \cite{M1990} (see also \cite{M1995}). The construction of this type of products is closely associated with the notion of matched pair of Hopf algebras. A matched pair of Hopf algebras in a symmetric category ${\sf C}$ is a quadruple $(A,H,\varphi_{A}, \phi_{H})$, where $A$ and $H$ are Hopf algebras, $(A,\varphi_{A})$ is a left $H$-module coalgebra, $(H, \phi_{H}) $ is a right $A$-module coalgebra and the following compatibility conditions hold:
\begin{gather*}
%\label{mp1}
\varphi_{A}\circ (H\otimes \eta_{A})= \varepsilon_{H}\otimes \eta_{A}, \\
%\label{mp2}
\varphi_{A}\circ (H\otimes \mu_{A})= \mu_{A}\circ (A\otimes \varphi_{A})\circ (\Psi \otimes A ), \\
%\label{mp3}
\phi_{H}\circ (\eta_{H}\otimes A)= \varepsilon_{A}\otimes \eta_{H}, \\
%\label{mp4}
\phi_{H}\circ (\mu_{H}\otimes A)=\mu_{H}\circ (\phi_{H}\otimes A)\circ (H\otimes \Psi), \\
%\label{mp5}
c_{A,H}\circ \psi=(\phi_{H}\otimes \varphi_{A})\circ \delta_{H\otimes A},
\end{gather*}
where $\Psi=( \varphi_{A}\otimes \phi_{H})\circ \delta_{H\otimes A}$.

If $(A,H,\varphi_{A}, \phi_{H})$ is a matched pair of Hopf algebras, the bicrossed product of Hopf algebras $A\bowtie H$
 is the tensor coalgebra with the product and antipode defined as follows:
 \begin{gather*}
 \mu_{A\bowtie H}=(\mu_{A}\otimes \mu_{H})\circ (A\otimes \Psi\otimes H),\\
 \lambda_{A\bowtie H}=\Psi\circ (\lambda_{H}\otimes \lambda_{A})\circ c_{A,H}.
 \end{gather*}

 Moreover, under the previous conditions, $A\bowtie H$ is a Hopf algebra if and only if $(A,H,\varphi_{A}, \phi_{H})$ is a matched pair of Hopf algebras.

It is well known that the construction of $A\bowtie H$ gives a complete answer to the factorization problem, i.e., describe all Hopf algebras $L$ for which, giving two Hopf algebras $A$ and $H$ such that there exist Hopf algebra monomorphisms $i_{A}\colon A\rightarrow L$ and $i_{H}\colon H\rightarrow L$, the morphism $\mu_{L}\circ (i_{A}\otimes i_{H})\colon A\otimes H\rightarrow L$ is an isomorphism. As we can see in \cite[Theorem 7.2.3]{M1995}, a Hopf algebra $L$ factorizes through the Hopf algebras $A$ and $H$ if and only if there exists a matched pair of Hopf algebras $(A,H,\varphi_{A}, \phi_{H})$ and a Hopf algebra isomorphism between $L$ and $A\bowtie H$.

An example of the previous bicrossed product is the one where $\phi_{H}= H\otimes\varepsilon_{A}$, i.e., $\phi_{H}$ is trivial. In this case $A\bowtie H$, denoted by $A\ltimes H$, is the semidirect (smash) product introduced by Molnar \cite{Molnar} in the cocommutative Hopf algebra setting. Note that for the semidirect product we have that $\Psi=(\varphi_{A}\otimes H)\circ (H\otimes c_{H,A})\circ (\delta_{H}\otimes A)$ and the following holds: If $(A,\varphi_{A})$ is a left $H$-module coalgebra, $A\ltimes H$ is a Hopf algebra if and only if $(A,\varphi_{A})$ is a left $H$-module algebra and
\begin{gather}\label{classcoc2}
	(\varphi_{A}\otimes H)\circ(H\otimes c_{H,A})\circ((c_{H,H}\circ\delta_{H})\otimes A)=(\varphi_{A}\otimes H)\circ(H\otimes c_{H,A})\circ(\delta_{H}\otimes A)
\end{gather}
holds, i.e., $(A,\varphi_{A})$ is in the cocommutativity class of $H$. Obviously, if $H$ is cocommutative this condition holds automatically. In the category of sets, $A\ltimes H$ is the semidirect product of groups because in the category ${\sf Set}$ Hopf algebras are groups. Finally, remember that in the category of vector spaces over a filed a Hopf algebra $L$ factorizes through a normal Hopf subalgebra~$A$
and a Hopf subalgebra $H$ if and only if $L$ is isomorphic as a Hopf algebra to a~semidirect product~${A\ltimes H}$.

On the other hand, a general notion of crossed module of Hopf algebras was given by Fr\'egier and Wagemann \cite{FW}, who considered two Hopf algebras $A$ and $H$, a morphism $\varphi_A\colon H\ot A\rightarrow A$ such that $(A, \varphi_A)$ is a left $H$-module algebra (coalgebra) and a Hopf algebra morphism $\partial\colon A\rightarrow H$ such that $\partial$ is a morphism of $H$-modules where $H$ carries the $H$-module structure given by the adjoint action for $H$, i.e.,
\begin{equation}
	\label{Freg1}
	\partial\co \varphi_A=\varphi_{H}^{{\rm ad}}\co (H\ot \partial),
\end{equation}
and the Peiffer identity holds, i.e.,
\begin{equation}
	\label{Freg2}
	\varphi_A\co (\partial\ot A)=\varphi_{A}^{{\rm ad}},
\end{equation}
where $\varphi^{{\rm ad}}_{H}$ and $\varphi^{{\rm ad}}_{A}$ denote the adjoint action for $H$ and $A$, respectively. In the same setting, Majid \cite{MA2} gives a notion of crossed module of Hopf algebras (see also \cite{Faria}) assuming (\ref{Freg1}) and~(\ref{Freg2}), the condition of morphism of left $H$-modules for the antipode of $A$ and (\ref{classcoc2}). The~main reason for assuming (\ref{classcoc2}) is that it is a necessary condition to assure that a crossed product of Hopf algebras is compatible with the tensor coproduct in the sense of the semidirect product. We also want to emphasize that, as was proved in \cite{TAC}, (\ref{classcoc2}) allows us to ensure that any Hopf algebra $H$ induces a crossed module of Hopf algebras $(H,H,{\rm id}_H)$, in a similar way to what happen in the group setting.
\end{Remark}

\begin{Theorem}\label{Prop: Weak-nuHTr}
Let $(H,m_{H},\Phi_{H})$ be an object in ${\sf wt}\textnormal{-}\sf{Post}\textnormal{-}\sf{Hopf}$ such that \eqref{classcocommH} holds. Then, the triple
$\overline{\mathbb{H}}=\bigl(H,\overline{H},\Phi_{H}\bigr)$
is an object in $\sf{HTr}$.
\end{Theorem}
\begin{proof}
Thanks to Theorem \ref{nuHbar}, it is enough to see that (iii) of Definition \ref{H-truss} holds. First of all, note that \begin{equation}\label{GammaPhi=mH}\Gamma_{H}^{\Phi_{H}}=m_{H}.\end{equation}

Indeed,
\begin{gather*}
 \Gamma_{H}^{\Phi_{H}}= \mu_{H}\circ((\mu_{H}\circ((\lambda_{H}\circ\Phi_{H})\otimes\Phi_{H})\circ\delta_{H})\otimes m_{H})\circ(\delta_{H}\otimes H)\\
 \hphantom{\Gamma_{H}^{\Phi_{H}}=}{}\ \text{\footnotesize\textnormal{(by associativity of $\mu_{H}$ and coassociativity of $\delta_{H}$)}}\\
 \hphantom{\Gamma_{H}^{\Phi_{H}}}{}= \mu_{H}\circ(((\lambda_{H}\ast {\rm id}_{H})\circ\Phi_{H})\otimes m_{H})\circ(\delta_{H}\otimes H)\ \text{\footnotesize\textnormal{(by (ii.1) of Definition \ref{WTPH})}}\\
 \hphantom{\Gamma_{H}^{\Phi_{H}}}{}= \mu_{H}\circ((\eta_{H}\circ\varepsilon_{H}\circ\Phi_{H})\otimes m_{H})\circ(\delta_{H}\otimes H)\ \text{\footnotesize\textnormal{(by \eqref{antipode})}}\\
 \hphantom{\Gamma_{H}^{\Phi_{H}}}{}= \mu_{H}\circ((\eta_{H}\circ\varepsilon_{H})\otimes m_{H})\circ(\delta_{H}\otimes H)\ \text{\footnotesize\textnormal{(by (ii.2) of Definition \ref{WTPH})}}\\
 \hphantom{\Gamma_{H}^{\Phi_{H}}}{}= m_{H}\ \text{\footnotesize\textnormal{(by (co)unit property)}}.
\end{gather*}

Therefore, we obtain
\begin{gather*}
 \mu_{H}\circ\bigl(\overline{\mu}_{H}\otimes\Gamma_{H}^{\Phi_{H}}\bigr)\circ(H\otimes c_{H,H}\otimes H)\circ(\delta_{H}\otimes H\otimes H)\\
 \quad{}= \mu_{H}\circ((\mu_{H}\circ(\Phi_{H}\otimes m_{H})\circ(\delta_{H}\otimes H))\otimes m_{H})\circ(H\otimes c_{H,H}\otimes H)\circ(\delta_{H}\otimes H\otimes H)\\
 \qquad{}\;\text{\footnotesize\textnormal{(by definition of $\overline{\mu}_{H}$ and \eqref{GammaPhi=mH})}}\\
 \quad{}= \mu_{H}\circ(\Phi_{H}\otimes(\mu_{H}\circ(m_{H}\otimes m_{H})\circ(H\otimes c_{H,H}\otimes H)\circ(\delta_{H}\otimes H\otimes H)))\circ(\delta_{H}\otimes H\otimes H)\\
 \qquad{}\;\text{\footnotesize\textnormal{(by coassociativity of $\delta_{H}$ and associativity of $\mu_{H}$)}}\\
 \quad{}= \overline{\mu}_{H}\circ(H\otimes\mu_{H})\ \text{\footnotesize\textnormal{(by (v) of Definition \ref{WTPH})}}.
 \tag*{\qed}
\end{gather*}
\renewcommand{\qed}{}
\end{proof}

As a consequence, if we set ${\sf wt}\textnormal{-}{\sf Post}\textnormal{-}{\sf Hopf}^{\star}$ to be the full subcategory of ${\sf wt}\textnormal{-}{\sf Post}\textnormal{-}{\sf Hopf}$ whose objects satisfy \eqref{classcocommH}, then there exists a functor $F\colon \textnormal{{\sf wt}-{\sf Post}-{\sf Hopf}}^{\star}\longrightarrow \sf{HTr}$ defined on objects by $F((H,m_{H},\Phi_{H}))=\overline{\mathbb{H}}$ and on morphisms by the identity. Note that $F$ is well defined on morphisms. Indeed, if $f\colon (H,m_{H},\Phi_{H})\rightarrow (B,m_{B},\Phi_{B})$ is a morphism in ${\sf wt}\textnormal{-}{\sf Post}\textnormal{-}{\sf Hopf}^{\star}$, then\looseness=1
\begin{gather*}
 f\circ\overline{\mu}_{H}= \mu_{B}\circ((f\circ\Phi_{H})\otimes(f\circ m_{H}))\circ(\delta_{H}\otimes H)\ \text{\footnotesize\textnormal{(by the condition of algebra morphism for $f$)}}\\
 \hphantom{f\circ\overline{\mu}_{H}}{}= \mu_{B}\circ(\Phi_{B}\otimes m_{B})\circ(((f\otimes f)\circ\delta_{H})\otimes f)\ \text{\footnotesize\textnormal{(by \eqref{Cond1 mor TPH} and \eqref{Cond2 mor TPH})}}\\
 \hphantom{f\circ\overline{\mu}_{H}}{}= \overline{\mu}_{B}\circ(f\otimes f)\ \text{\footnotesize\textnormal{(by the condition of coalgebra morphism for $f$)}}.\qedhere
\end{gather*}

Now, our aim is to construct a functor in the opposite sense, that is to say, we want to construct a functor from $\sf{HTr}$ to ${\sf wt}\textnormal{-}\sf{Post}$-${\sf Hopf}^{\star}$.
\begin{Theorem}\label{functorG}
Let $\mathbb{H}=(H_{1},H_{2},\sigma_{H})$ be an object in $\sf{HTr}$ such that the condition
\begin{gather}
\bigl(\Gamma_{H_{1}}^{\sigma_{H}}\otimes H\bigr)\hspace{-0.5pt}\circ\hspace{-0.5pt} (H\otimes c_{H,H})\hspace{-0.5pt} \circ\hspace{-0.5pt} ((c_{H,H}\circ\delta_{H})\otimes H)
=\bigl(\Gamma_{H_{1}}^{\sigma_{H}}\hspace{-0.5pt}\otimes\hspace{-0.5pt} H\bigr)\hspace{-0.5pt}\circ\hspace{-0.5pt}(H\otimes c_{H,H})\circ(\delta_{H}\hspace{-0.5pt}\otimes\hspace{-0.5pt} H),\!\!\!\label{classcocomGH}
\end{gather}
that is,
\[
\includegraphics[scale=3]{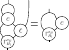}
\]
\noindent holds. Under these hypothesis, $\bigl(H_{1},\Gamma_{H_{1}}^{\sigma_{H}},\sigma_{H}\bigr)$
is an object in ${\sf wt}\textnormal{-}\sf{Post}\textnormal{-}\sf{Hopf}^{\star}$.
\end{Theorem}
\begin{proof} Let us see that the triple $\bigl(H_{1},\Gamma_{H_{1}}^{\sigma_{H}},\sigma_{H}\bigr)$ satisfies conditions (i)--(v) of Definition \ref{WTPH}. Thanks to Hopf truss' axioms of Definition \ref{H-truss}, $\sigma_{H}$ is a coalgebra morphism, so (ii) of Definition~\ref{WTPH} holds. Moreover, $\bigl(H_{1}, \Gamma_{H_{1}}^{\sigma_{H}}\bigr)$ is a non-unital left $H_{2}$-module algebra. Therefore, due to this fact and \eqref{mu2Htruss}, conditions (iv) and (v) of Definition \ref{WTPH} hold. Condition (iii) follows by~\eqref{mu2Htruss} and~\eqref{cocycle truss 2}. So, to conclude the proof it is enough to see that $\Gamma_{H_{1}}^{\sigma_{H}}$ is a coalgebra morphism. Indeed,
\begin{gather*}
 \delta_{H}\circ\Gamma_{H_{1}}^{\sigma_{H}} \\
 =\bigl(\mu_{H}^{1}\otimes\mu_{H}^{1}\bigr)\circ(H\otimes c_{H,H}\otimes H) \\
 \quad{} \circ\bigl((\delta_{H}\circ\lambda_{H}\circ\sigma_{H})\otimes\bigl(\bigl(\mu_{H}^{2}\otimes\mu_{H}^{2}\bigr)\circ(H\otimes c_{H,H}\otimes H)\circ(\delta_{H}\otimes\delta_{H})\bigr)\bigr)\circ(\delta_{H}\otimes H)\\
 \quad{}\;\text{\footnotesize\textnormal{\big(by the condition of coalgebra morphism for $\mu_{H}^{k}$, $k=1,2$\big)}} \\
 =\bigl(\mu_{H}^{1}\otimes\mu_{H}^{1}\bigr)\circ(H\otimes c_{H,H}\otimes H)\\
 \quad{}\circ\bigl((c_{H,H}\circ(\lambda_{H}\otimes\lambda_{H})\circ(\sigma_{H}\otimes\sigma_{H})\circ\delta_{H})\!\otimes\!\bigl(\bigl(\mu_{H}^{2}\otimes\mu_{H}^{2}\bigr)
\circ(H\otimes c_{H,H}\otimes H)\circ (\delta_{H}\otimes\delta_{H})\bigr)\bigr)\\
 \quad{}\circ(\delta_{H}\otimes H)\ \text{\footnotesize\textnormal{(by \eqref{a-antip2} and the condition of coalgebra morphism for $\sigma_{H}$)}} \\ =\bigl(\bigl(\mu_{H}^{1}\circ\bigl((\lambda_{H}\circ\sigma_{H})\otimes\mu_{H}^{2}\bigr)\bigr)\otimes\bigl(\mu_{H}^{1}\circ\bigl(\bigl(\lambda_{H}\circ\sigma_{H}\bigr)\otimes\mu_{H}^{2}\bigr)\bigr)\bigr)\\
 \quad{}\circ(H\otimes H\otimes c_{H,H}\otimes H\otimes H) \\
 \quad{}\circ(((H\otimes c_{H,H})\circ(c_{H,H}\otimes H)\circ(H\otimes\delta_{H})\circ\delta_{H})\otimes H\otimes H\otimes H)\circ(H\otimes c_{H,H}\otimes H)\\
 \quad{}\circ(\delta_{H}\otimes\delta_{H})\ \text{\footnotesize\textnormal{(by coassociativity of $\delta_{H}$ and naturality of $c$)}} \\
 =\bigl(H\otimes\bigl(\mu_{H}^{1}\circ\bigl((\lambda_{H}\circ\sigma_{H})\otimes\mu_{H}^{2}\bigr)\bigr)\bigr)\\
 \quad{}\circ\bigl(\bigl(\bigl(\Gamma_{H_{1}}^{\sigma_{H}}\otimes H\bigr)\circ(H\otimes c_{H,H})\circ((c_{H,H}\circ\delta_{H})\otimes H)\bigr)\otimes H\otimes H\bigr)\\
 \quad{}\circ(H\otimes c_{H,H}\otimes H)\circ(\delta_{H}\otimes\delta_{H}) \ \text{\footnotesize\textnormal{\big(by naturality of $c$ and definition of $\Gamma_{H_{1}}^{\sigma_{H}}$\big)}} \\
 =\bigl(H\otimes\bigl(\mu_{H}^{1}\circ\bigl((\lambda_{H}\circ\sigma_{H})\otimes\mu_{H}^{2}\bigr)\bigr)\bigr)\circ\bigl(\bigl(\bigl(\Gamma_{H_{1}}^{\sigma_{H}}\otimes H\bigr)\circ(H\otimes c_{H,H})\circ(\delta_{H}\otimes H)\bigr)\otimes H\otimes H\bigr)\\
 \quad{}\circ(H\otimes c_{H,H}\otimes H)\circ(\delta_{H}\otimes\delta_{H}) \ \text{\footnotesize\textnormal{(by \eqref{classcocomGH})}}\\
 =\bigl(\Gamma_{H_{1}}^{\sigma_{H}}\otimes\Gamma_{H_{1}}^{\sigma_{H}}\bigr)\circ(H\otimes c_{H,H}\otimes H)\circ(\delta_{H}\otimes\delta_{H})\ \text{\footnotesize\textnormal{(by coassociativity of $\delta_{H}$ and naturality of $c$)}},
\end{gather*}
and it is straightforward to see that $\varepsilon_{H}\circ\Gamma_{H_{1}}^{\sigma_{H}}=\varepsilon_{H}\otimes\varepsilon_{H}$, so (i) of Definition \ref{WTPH} also holds.
\end{proof}

\begin{Remark}
When $\sf{C}$ is symmetric, note that \eqref{classcocomGH} means that $\bigl(H_{1},\Gamma_{H_{1}}^{\sigma_{H}}\bigr)$ is in the cocommutativity class of $H_{2}$.
\end{Remark}
From now on, let us denote by~$\sf{HTr}^{\star}$ to the full subcategory of $\sf{HTr}$ whose objects satisfy condition~\eqref{classcocomGH}. Therefore, we can interpret the previous theorem as follows: There exists a functor
$G\colon \sf{HTr}^{\star}\longrightarrow \textnormal{{\sf wt}-{\sf Post}-{\sf Hopf}}^{\star}$
acting on objects by $G(\mathbb{H})=\bigl(H_{1},\Gamma_{H_{1}}^{\sigma_{H}},\sigma_{H}\bigr)$ and on morphisms by the identity. Note that $G$ is well defined on morphisms because, if $f\colon \mathbb{H}=(H_{1},H_{2},\sigma_{H})\rightarrow \mathbb{B}=(B_{1},B_{2},\sigma_{B})$ is a morphism of Hopf trusses, then
\begin{gather}
 f\circ\Gamma_{H_{1}}^{\sigma_{H}} = \mu_{B}^{1}\circ\bigl((f\circ\lambda_{H}\circ\sigma_{H})\otimes\bigl(f\circ \mu_{H}^{2}\bigr)\bigr)\circ(\delta_{H}\otimes H)\nonumber\\
\hphantom{f\circ\Gamma_{H_{1}}^{\sigma_{H}}=}{}
\text{\footnotesize\textnormal{(by the condition of algebra morphism for $f\colon H_{1}\rightarrow B_{1}$)}}\nonumber\\ \hphantom{f\circ\Gamma_{H_{1}}^{\sigma_{H}}}{}
= \mu_{B}^{1}\circ\bigl((\lambda_{B}\circ\sigma_{B})\otimes\mu_{B}^{2}\bigr)\circ(((f\otimes f)\circ\delta_{H})\otimes f)\nonumber\\ \hphantom{f\circ\Gamma_{H_{1}}^{\sigma_{H}}=}{}
 \text{\footnotesize\textnormal{(by the condition of algebra morphism for $f\colon H_{2}\rightarrow B_{2}$, \eqref{morant} and \eqref{Cond mor truss})}}\nonumber\\ \hphantom{f\circ\Gamma_{H_{1}}^{\sigma_{H}}}{}
= \Gamma_{B_{1}}^{\sigma_{B}}\circ(f\otimes f)\ \text{\footnotesize\textnormal{(by the condition of coalgebra morphism for $f$)}}.\label{gammatrussf}
\end{gather}
Note also that if $(H,m_{H},\Phi_{H})$ is an object in ${\sf wt}\textnormal{-}{\sf Post}$-${\sf Hopf}^{\star}$, the Hopf truss $F((H,m_{H},\Phi_{H}))=\overline{\mathbb{H}}$ defined in Theorem \ref{Prop: Weak-nuHTr} belongs to the category $\sf{HTr}^{\star}$, so $F$ admits a restriction from ${\sf wt}\textnormal{-}{\sf Post}$-${\sf Hopf}^{\star}$ to $\sf{HTr}^{\star}$.

The following theorem is one of the most important of this section, and this is the generalization to the braided monoidal framework of \cite[Theorem 6.7]{Wang}. Note also that, in contrast with the result cited before, weak twisted post-Hopf algebras and Hopf trusses are not required to be cocommutative.
\begin{Theorem}\label{th-iso-wtph-htr}
The categories ${\sf wt}\textnormal{-}{\sf Post}\textnormal{-}{\sf Hopf}^{\star}$ and $\sf{HTr}^{\star}$ are isomorphic.
\end{Theorem}
\begin{proof}
On the one side, consider $(H,m_{H},\Phi_{H})$ an object in ${\sf wt}\textnormal{-}{\sf Post}\textnormal{-}{\sf Hopf}^{\star}$. By \eqref{GammaPhi=mH}, we obtain that
\begin{align*}
&(G\circ F)(H,m_{H},\Phi_{H})=G\bigl(\bigl(H,\overline{H},\Phi_{H}\bigr)\bigr)=\bigl(H,\Gamma_{H}^{\Phi_{H}},\Phi_{H}\bigr)=(H,m_{H},\Phi_{H}),
\end{align*}
and therefore, $G\circ F={\sf id}_{{\sf wt}\textnormal{-}{\sf Post}\textnormal{-}{\sf Hopf}^{\star}}$.

On the other side, let $\mathbb{H}=(H_{1},H_{2},\sigma_{H})$ be an object in $\sf{HTr}^{\star}$. We have the following:
\begin{align*}
&(F\circ G)(\mathbb{H})=F\bigl(\bigl(H_{1},\Gamma_{H_{1}}^{\sigma_{H}},\sigma_{H}\bigr)\bigr)=\bigl(H_{1},\overline{H_{1}},\sigma_{H}\bigr),
\end{align*}
where, by \eqref{mu2Htruss}, $\overline{\mu_{H_{1}}}=\mu_{H}^{2}$. Therefore, we conclude that $\overline{H_{1}}=H_{2}$, and so $F\circ G={\sf id}_{\sf{HTr}^{\star}}$.
\end{proof}

\begin{Corollary}\label{cor-iso-cocwtph-cochtr}
Categories $\sf{coc}\textnormal{-}{\sf wt}\textnormal{-}{\sf Post}$-$\sf{Hopf}$ and $\sf{coc}\textnormal{-}{\sf HTr}$ are isomorphic.
\end{Corollary}
\begin{proof}
By previous theorem, this corollary is direct taking into account that, under cocommutativity, \eqref{classcocommH} and \eqref{classcocomGH} always hold.
\end{proof}

Having reached the fundamental outcome of this section, the following results will try to generalize some additional properties that can be found in \cite{Wang} to the monoidal context.
\begin{Lemma}\label{Phiidemp}
Let $(H,m_{H},\Phi_{H})$ be an object in ${\sf wt}\textnormal{-}\sf{Post}\textnormal{-}\sf{Hopf}$ satisfying that $\Phi_{H}\circ\eta_{H}=\eta_{H}$. Then, $\Phi_{H}$ is idempotent.
\end{Lemma}
\begin{proof}The proof follows by
\begin{align*}
\Phi_{H}={}&\overline{\mu}_{H}\circ(H\otimes\eta_{H})\ \text{\footnotesize\textnormal{(by \eqref{cond1WTPHA})}}\\
={}&\overline{\mu}_{H}\circ\left(H\otimes(\Phi_{H}\circ\eta_{H})\right)\ \text{\footnotesize\textnormal{(by $\Phi_{H}\circ\eta_{H}=\eta_{H}$)}}\\
={}&\Phi_{H}\circ\overline{\mu}_{H}\circ(H\otimes\eta_{H})\ \text{\footnotesize\textnormal{(by (iii) of Definition \ref{WTPH})}}\\
={}&\Phi_{H}\circ\Phi_{H}\ \text{\footnotesize\textnormal{(by \eqref{cond1WTPHA})}}.\tag*{\qed}
\end{align*}
\renewcommand{\qed}{}
\end{proof}

As a consequence, if $(H,m_{H},\Phi_{H})$ is an object in ${\sf t}\textnormal{-}\sf{Post}\textnormal{-}\sf{Hopf}$, then $\Phi_{H}$ is idempotent.
\begin{Lemma}
Let $(H,m_{H},\Phi_{H})$ be an object in ${\sf t}\textnormal{-}\sf{Post}\textnormal{-}\sf{Hopf}$. The equality
\begin{equation}\label{mconPHI}
m_{H}=m_{H}\circ(\Phi_{H}\otimes H).
\end{equation}
holds.
\end{Lemma}
\begin{proof}The proof follows by
\begin{align*}
m_{H}={}&m_{H}\circ(H\otimes(m_{H}\circ(\eta_{H}\otimes H)))\ \text{\footnotesize\textnormal{(by \eqref{mHprop2})}}\\
={}&m_{H}\circ((\overline{\mu}_{H}\circ(H\otimes\eta_{H}))\otimes H)\ \text{\footnotesize\textnormal{(by (iv) of Definition \ref{WTPH})}}\\
={}&m_{H}\circ(\Phi_{H}\otimes H)\ \text{\footnotesize\textnormal{(by \eqref{cond1WTPHA})}}.\tag*{\qed}
\end{align*}
\renewcommand{\qed}{}
\end{proof}

Consider $(H,m_{H},\Phi_{H})$ a twisted post-Hopf algebra and define the following morphism:
\begin{equation*}%\label{SHnococom}
S_{H}\coloneqq ( b_{H}(K)\otimes H)\circ((\lambda_{H}\circ\Phi_{H})\otimes\beta_{H})\circ c_{H,H}\circ\delta_{H},
\end{equation*}
i.e.,
\[
\includegraphics[scale=3]{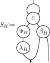}
\]

Note that, when $H$ is cocommutative, the previous morphism becomes the following:
\begin{equation}
\label{SHcocom}
S_{H}=( b_{H}(K)\otimes H)\circ((\lambda_{H}\circ\Phi_{H})\otimes\beta_{H})\circ\delta_{H}.
\end{equation}

So, next results are going to be devoted to studying under what conditions $S_{H}$ satisfies that
\begin{equation}\label{antipodeSh}
S_{H}\barast {\rm id}_{H}=\varepsilon_{H}\otimes\eta_{H}={\rm id}_{H}\barast S_{H},
\end{equation}
where $\overline{\ast}$ denotes the convolution product in $\Hom\bigl(H,\overline{H}\bigr)$.

\begin{Lemma}
Let $(H,m_{H},\Phi_{H})$ be an object in $\sf{coc}\textnormal{-}\sf{t}\textnormal{-}\sf{Post}\textnormal{-}\sf{Hopf}$. Then,
\begin{equation}\label{relation lambdaH and SH}
\lambda_{H}\circ\Phi_{H}=m_{H}\circ (H\otimes S_{H})\circ\delta_{H}.
\end{equation}
\end{Lemma}
\begin{proof} By \eqref{lem1prop2} and the fact that $\beta_{H}$ is the inverse of $\alpha_{H}$ for the convolution in $\Hom(H,H^{\ast}\otimes H)$, then
	\begin{gather*}
 m_{H} \circ (H\otimes S_{H})\circ\delta_{H}\\
 \quad{}= m_{H}\circ (H\otimes ((b_{H}(K)\otimes H)\circ ((\lambda_{H}\circ\Phi_{H})\otimes\beta_{H})\circ\delta_{H}))\circ\delta_{H}\ \text{\footnotesize\textnormal{(by \eqref{SHcocom})}}\\
 \quad{}= (b_{H}(K)\otimes H)\circ (H\otimes\alpha_{H})\circ c_{H,H}\circ (H\otimes ((b_{H}(K)\otimes H)\circ ((\lambda_{H}\circ\Phi_{H})\otimes\beta_{H})\circ\delta_{H}))\\
 \qquad{}\circ\delta_{H}\ \text{\footnotesize\textnormal{(by \eqref{lem1prop2})}}\\
 \quad{}= ((b_{H}(K)\circ ((\lambda_{H}\circ \Phi_{H})\otimes H^{\ast}))\otimes (b_{H}(K)\circ c_{H^{\ast},H})\otimes H)\circ (H\otimes c_{H^{\ast},H^{\ast}}\otimes c_{H,H})\\
 \qquad{}\circ (H\otimes H^{\ast}\otimes c_{H,H^{\ast}}\otimes H)\circ (((c_{H^{\ast},H}\otimes H)\circ (H^{\ast}\otimes c_{H,H})\circ (\alpha_{H}\otimes H))\otimes \beta_{H})\\
 \qquad{}\circ (H\otimes\delta_{H})\circ\delta_{H}\ \text{\footnotesize\textnormal{(by naturality of $c$)}}\\
 \quad{}= ((b_{H}(K)\circ ((\lambda_{H}\circ\Phi_{H})\otimes H^{\ast}))\otimes H)\circ (H\otimes ((H^{\ast}\otimes b_{H}(K)\otimes H)\circ (\beta_{H}\otimes\alpha_{H})))\\
 \qquad{}\circ (H\otimes c_{H,H})\circ (c_{H,H}\otimes H)\circ (H\otimes\delta_{H})\circ\delta_{H}\ \text{\footnotesize\textnormal{(by naturality of $c$)}}\\
 \quad{}= ((b_{H}(K)\circ ((\lambda_{H}\circ\Phi_{H})\otimes H^{\ast}))\otimes H)\circ (H\otimes ((H^{\ast}\otimes b_{H}(K)\otimes H)\circ (\beta_{H}\otimes\alpha_{H})\circ\delta_{H}))\\
 \qquad{}\circ\delta_{H}\ \text{\footnotesize\textnormal{(by naturality of $c$ and cocommutativity and coassociativity of $\delta_{H}$)}}\\
 \quad{}= ((b_{H}(K)\circ ((\lambda_{H}\circ\Phi_{H})\otimes H^{\ast}))\otimes H)\circ (H\otimes (\varepsilon_{H}\otimes a_{H}(K)))\\
 \qquad{}\circ\delta_{H}\ \text{\footnotesize\textnormal{(by (vii) of Definition \ref{WTPH})}}\\
 \quad{}= \lambda_{H}\circ\Phi_{H}\ \text{\footnotesize\textnormal{(by counit property and the adjunction properties)}}.\tag*{\qed}
	\end{gather*}
	\renewcommand{\qed}{}
\end{proof}

\begin{Theorem}
If $(H,m_{H},\Phi_{H})$ is an object in $\sf{coc}\textnormal{-}\sf{t}\textnormal{-}\sf{Post}\textnormal{-}\sf{Hopf}$, then
\begin{equation}\label{convolution one side}{\rm id}_{H}\barast S_{H}=\varepsilon_{H}\otimes\eta_{H}.\end{equation}
\end{Theorem}
\begin{proof} We have that
\begin{align*}
{\rm id}_{H}\barast S_{H}={}&\mu_{H}\circ(\Phi_{H}\otimes (m_{H}\circ(H\otimes S_{H})\circ\delta_{H}))\circ\delta_{H}\ \text{\footnotesize\textnormal{(by coassociativity of $\delta_{H}$)}}\\
={}&\mu_{H}\circ(H\otimes \lambda_{H})\circ(\Phi_{H}\otimes\Phi_{H})\circ\delta_{H}\ \text{\footnotesize\textnormal{(by \eqref{relation lambdaH and SH})}}\\
={}&({\rm id}_{H}\ast\lambda_{H})\circ\Phi_{H}\ \text{\footnotesize\textnormal{(by (ii.1) of Definition \ref{WTPH})}}\\
={}&\varepsilon_{H}\otimes\eta_{H}\ \text{\footnotesize\textnormal{(by \eqref{antipode} and (ii.2) of Definition \ref{WTPH})}}.\tag*{\qed}
\end{align*}
\renewcommand{\qed}{}
\end{proof}

\begin{Lemma}
Let $(H,m_{H},\Phi_{H})$ be an object in $\sf{coc}\textnormal{-}\sf{t}\textnormal{-}\sf{Post}\textnormal{-}\sf{Hopf}$. The morphism
\[\tilde{\alpha}_{H}\coloneqq \left( b_{H}(K)\otimes H\right)\circ(H\otimes\alpha_{H})\]
is a coalgebra morphism.
\end{Lemma}
\begin{proof}
On the one side, by \eqref{lem1prop1} and \eqref{ccb}, $\tilde{\alpha}_{H}=m_{H}\circ c_{H,H}$. On the other side, according to~(i.1) of Definition \ref{WTPH}, \eqref{ccb} and naturality of $c$, both $m_{H}$ and $c_{H,H}$ are coalgebra morphisms. Then, $m_{H}\circ c_{H,H}$ is a coalgebra morphism and, as a result, $\tilde{\alpha}_{H}$ is a coalgebra morphism.
\end{proof}

Let us define $\tilde{\beta}_{H}\coloneqq\left( b_{H}(K)\otimes H\right)\circ(H\otimes\beta_{H}),$
where $\beta_{H}$ is the inverse of $\alpha_{H}$ for the convolution product in $\Hom(H,H^{\ast}\otimes H)$. For usual post-Hopf algebras (in other words, when $\Phi_{H}={\rm id}_{H}$) it was necessary to require $\tilde{\beta}_{H}$ to be a morphism of coalgebras to conclude that $S_{H}$ is the inverse of ${\rm id}_{H}$ for the convolution $\overline{\ast}$ \big(when $\Phi_{H}={\rm id}_{H}$ note that $\overline{H}$ is unital with unit $\eta_{H}$ and, as a~consequence, $\Hom\bigl(H,\overline{H}\bigr)$ is also unital with unit $\eta_{H}\circ\varepsilon_{H}$ for the convolution $\barast$, see~\mbox{\cite[Lemma~3.13 and Theorem~3.14]{FGR}}\big). A natural question that arises at this point is whether such a condition implies that $S_{H}$ satisfies \eqref{antipodeSh} in general, not only when $\Phi_{H}={\rm id}_{H}$. This question will be solved in Theorem \ref{iffPhi=id}.
\begin{Theorem}
Let $(H,m_{H},\Phi_{H})$ be an object in $\sf{coc}\textnormal{-}\sf{t}\textnormal{-}\sf{Post}\textnormal{-}\sf{Hopf}$. If $\tilde{\beta}_{H}$ is a coalgebra morphism, then $S_{H}$ is a coalgebra morphism.
\end{Theorem}
\begin{proof}
First of all, we will see that $\varepsilon_{H}\circ S_{H}=\varepsilon_{H}$,
\begin{align*}
\varepsilon_{H}\circ S_{H}
={}&\varepsilon_{H}\circ\tilde{\beta}_{H}\circ((\lambda_{H}\circ\Phi_{H})\otimes H)\circ\delta_{H}\ \text{\footnotesize\textnormal{\big(by \eqref{SHcocom} and definition of $\tilde{\beta}_{H}$\big)}}\\
={}&((\varepsilon_{H}\circ\lambda_{H}\circ\Phi_{H})\otimes\varepsilon_{H})\circ\delta_{H}\ \text{\footnotesize\textnormal{\big(by $\tilde{\beta}_{H}$ coalgebra morphism\big)}}\\
={}&\varepsilon_{H}\ \text{\footnotesize\textnormal{(by \eqref{u-antip2}, (ii.2) of Definition \ref{WTPH} and counit property)}}.
\end{align*}

On the other hand, we have that
\begin{gather*}
 \delta_{H}\circ S_{H}\\
 = \delta_{H}\circ\tilde{\beta}_{H}\circ((\lambda_{H}\circ\Phi_{H})\otimes H)\circ\delta_{H}\ \text{\footnotesize\textnormal{\big(by \eqref{SHcocom} and definition of $\tilde{\beta}_{H}$\big)}}\\
 = \bigl(\tilde{\beta}_{H}\otimes \tilde{\beta}_{H}\bigr)\circ(H\otimes c_{H,H}\otimes H)\circ((\delta_{H}\circ\lambda_{H}\circ\Phi_{H})\otimes\delta_{H})\circ\delta_{H}\ \text{\footnotesize\textnormal{\big(by $\tilde{\beta}_{H}$ coalgebra morphism\big)}}\\
 = \bigl(\tilde{\beta}_{H}\otimes \tilde{\beta}_{H}\bigr)\circ(H\otimes c_{H,H}\otimes H)\circ((c_{H,H}\circ(\lambda_{H}\otimes\lambda_{H})\circ(\Phi_{H}\otimes\Phi_{H})\circ\delta_{H} )\otimes\delta_{H})\circ\delta_{H}\\
 \quad{} \text{\footnotesize\textnormal{(by \eqref{a-antip2} and (ii.1) of Definition \ref{WTPH})}}\\
 = \bigl(\bigl(\tilde{\beta}_{H}\circ((\lambda_{H}\circ\Phi_{H})\otimes H)\bigr)\otimes\bigl(\tilde{\beta}_{H}\circ((\lambda_{H}\circ\Phi_{H})\otimes H)\bigr)\bigr)\\
 \quad{} \circ(((H\otimes c_{H,H})\circ(c_{H,H}\otimes H)\circ(H\otimes\delta_{H})\circ\delta_{H})\otimes H)\circ\delta_{H}\\
 \quad{} \text{\footnotesize\textnormal{(by naturality of $c$ and coassociativity of $\delta_{H}$)}}\\
 = \bigl(S_{H}\otimes\bigl(\tilde{\beta}_{H}\circ((\lambda_{H}\circ\Phi_{H})\otimes H)\bigr)\bigr)\circ((c_{H,H}\circ\delta_{H})\otimes H)\circ\delta_{H}\\
 \quad{} \text{\footnotesize\textnormal{(by naturality of $c$, definition of $\tilde{\beta}_{H}$ and \eqref{SHcocom})}}\\
 = (S_{H}\otimes S_{H})\circ\delta_{H}\ \text{\footnotesize\textnormal{\big(by cocommutativity and coassociativity of $H$, definition of $\tilde{\beta}_{H}$ and \eqref{SHcocom}\big)}}.\tag*{\qed}
\end{gather*}
\renewcommand{\qed}{}
\end{proof}

\begin{Corollary}
Let $(H,m_{H},\Phi_{H})$ be an object in $\sf{coc}\textnormal{-}\sf{t}\textnormal{-}\sf{Post}\textnormal{-}\sf{Hopf}$. If $\tilde{\beta}_{H}$ is a coalgebra morphism, then
\begin{equation}\label{SH2=PHI}
S_{H}\circ S_{H}=\Phi_{H}.
\end{equation}
\end{Corollary}
\begin{proof} The equality follows by
\begin{align*}
S_{H}\circ S_{H}
={}&\overline{\mu}_{H}\circ((\eta_{H}\circ\varepsilon_{H})\otimes(S_{H}\circ S_{H}))\circ\delta_{H}\ \text{\footnotesize\textnormal{(by counit property and \eqref{muetaid})}}\\
={}&\overline{\mu}_{H}\circ((\overline{\mu}_{H}\circ(H\otimes S_{H})\circ\delta_{H})\otimes(S_{H}\circ S_{H}))\circ\delta_{H}\ \text{\footnotesize\textnormal{(by \eqref{convolution one side})}}\\
={}&\overline{\mu}_{H}\circ(H\otimes(\overline{\mu}_{H}\circ(S_{H}\otimes(S_{H}\circ S_{H}))\circ\delta_{H}))\circ\delta_{H}\\
&{}\text{\footnotesize\textnormal{(by associativity of $\overline{\mu}_{H}$ and coassociativity of $\delta_{H}$)}}\\
={}&\overline{\mu}_{H}\circ(H\otimes(({\rm id}_{H}\barast S_{H})\circ S_{H}))\circ\delta_{H}\ \text{\footnotesize\textnormal{(by the condition of coalgebra morphism for $S_{H}$)}}\\
={}&\Phi_{H}\ \text{\footnotesize\textnormal{(by \eqref{convolution one side}, $S_{H}$ coalgebra morphism, counit property and $\eqref{cond1WTPHA}$)}}.\tag*{\qed}
\end{align*}
\renewcommand{\qed}{}
\end{proof}

Note that, under the conditions of the previous corollary, thanks to \eqref{SH2=PHI} and the idempotent character of $\Phi_{H}$, we can obtain that
\begin{align}
\Phi_{H}\circ S_{H}\circ S_{H}\circ \Phi_{H}={}&\Phi_{H}\circ\Phi_{H}\circ\Phi_{H}\ \text{\footnotesize\textnormal{(by \eqref{SH2=PHI})}}\nonumber\\
={}&\Phi_{H}\ \text{\footnotesize\textnormal{(by the idempotent character of $\Phi_{H}$)}}.\label{impliSH2=PHI}
\end{align}

Although it may seem that $S_{H}$ is going to satisfy \eqref{antipodeSh} for the convolution in $\Hom\bigl(H,\overline{H}\bigr)$, it doesn't always happen. In fact, we will see in the following result that $S_{H}\barast {\rm id}_{H}=\varepsilon_{H}\otimes\eta_{H}$ if and only if $\Phi_{H}={\rm id}_{H}$, that is to say, taking into account Example \ref{ExampleTwPHA}, $S_{H}\barast {\rm id}_{H}=\varepsilon_{H}\otimes\eta_{H}$ if and only if $(H,m_{H})$ is a post-Hopf algebra.
\begin{Theorem}\label{iffPhi=id}
Let $(H,m_{H},\Phi_{H})$ be an object in $\sf{coc}\textnormal{-}\sf{t}\textnormal{-}\sf{Post}\textnormal{-}\sf{Hopf}$ such that $\tilde{\beta}_{H}$ is a coalgebra morphism.
The following facts are equivalent:
\begin{itemize}\itemsep=0pt
\item[$(i)$]$\Phi_{H}={\rm id}_{H}$.
\item[$(ii)$] $S_{H}\barast {\rm id}_{H}=\varepsilon_{H}\otimes\eta_{H}.$
\end{itemize}
\end{Theorem}
\begin{proof}
At first we will suppose that $\Phi_{H}={\rm id}_{H}$. To prove (ii) it is enough to follow the same arguments as in \cite[Lemma 3.13]{FGR} for post-Hopf algebras.

Assume now that $S_{H}\barast {\rm id}_{H}=\varepsilon_{H}\otimes\eta_{H}$. This implies that
\begin{align}
\overline{\mu}_{H}\circ(H\otimes (S_{H}\barast {\rm id}_{H}))\circ\delta_{H}
={}&\overline{\mu}_{H}\circ(H\otimes (\eta_{H}\circ\varepsilon_{H}))\circ\delta_{H}\nonumber\\
={}&\Phi_{H}\ \text{\footnotesize\textnormal{(by \eqref{cond1WTPHA})}}.\label{teorema1}
\end{align}

So we conclude the following:
\begin{align*}
\Phi_{H}={}&\overline{\mu}_{H}\circ(H\otimes (S_{H}\barast {\rm id}_{H}))\circ\delta_{H}\ \text{\footnotesize\textnormal{(by \eqref{teorema1})}}\\
={}&\overline{\mu}_{H}\circ(({\rm id}_{H}\barast S_{H})\otimes H)\circ\delta_{H}\ \text{\footnotesize\textnormal{(by coassociativity of $\delta_{H}$ and associativity of $\overline{\mu}_{H}$)}}\\
={}&\overline{\mu}_{H}\circ((\eta_{H}\circ\varepsilon_{H})\otimes H)\circ\delta_{H}\ \text{\footnotesize\textnormal{(by \eqref{convolution one side})}}\\
={}&{\rm id}_{H}\ \text{\footnotesize\textnormal{(by counit property and \eqref{muetaid})}}.\tag*{\qed}
\end{align*}
\renewcommand{\qed}{}
\end{proof}

So, is it possible to get a Hopf algebra structure from a twisted post-Hopf algebra? To solve this question we are going to assume that our base category $\sf{C}$ admits split idempotents, that is to say, if $q_{X}\colon X\rightarrow X$ is an idempotent morphism in {\sf C}, then there exists an object $I(q_{X})\in\sf{C}$, an epimorphism $p_{X}\colon X\rightarrow I(q_{X})$ and a monomorphism $i_{X}\colon I(q_{X})\rightarrow X$ such that~${i_{X}\circ p_{X}=q_{X}}$ and~${p_{X}\circ i_{X}={\rm id}_{I(q_{X})}}$. The categories satisfying this property constitute a broad class that includes, among others, the categories with epi-monic decomposition for morphisms and categories with equalizers or coequalizers.

Then, when $(H,m_{H},\Phi_{H})$ is an object in ${\sf wt}\textnormal{-}\sf{Post}\textnormal{-}\sf{Hopf}$ satisfying that $\Phi_{H}\circ\eta_{H}=\eta_{H}$, we have proved in Lemma \ref{Phiidemp} that $\Phi_{H}$ is idempotent, therefore, there exists an object $I(\Phi_{H})\in\sf{C}$, an epimorphism $p_{H}\colon H\rightarrow I(\Phi_{H})$ and a monomorphism $i_{H}\colon I(\Phi_{H})\rightarrow H$ verifying that
\begin{equation}\label{iconp}
i_{H}\circ p_{H}=\Phi_{H}
\end{equation}
and
\begin{equation}\label{pconi}
p_{H}\circ i_{H}={\rm id}_{I(\Phi_{H})}.
\end{equation}

Note that \eqref{iconp} and \eqref{pconi} imply that
\begin{equation}\label{Phiconi}
\Phi_{H}\circ i_{H}=i_{H},\qquad p_{H}\circ \Phi_{H}=p_{H}.
\end{equation}
\begin{Theorem}\label{bialgIphi}
If $(H,m_{H},\Phi_{H})$ is an object in ${\sf t}\textnormal{-}\sf{Post}\textnormal{-}\sf{Hopf}$ satisfying \eqref{classcocommH}, then \[I(\Phi_{H})=\bigl(I(\Phi_{H}),\eta_{I(\Phi_{H})},\mu_{I(\Phi_{H})},\varepsilon_{I(\Phi_{H})},\delta_{I(\Phi_{H})}\bigr)\] is a bialgebra in {\sf C} with
\begin{alignat*}{3}
&\mu_{I(\Phi_{H})}\coloneqq p_{H}\circ\overline{\mu}_{H}\circ (i_{H}\otimes i_{H}),\qquad&&
\eta_{I(\Phi_{H})}\coloneqq p_{H}\circ \eta_{H},\\
&\delta_{I(\Phi_{H})}\coloneqq (p_{H}\otimes p_{H})\circ\delta_{H}\circ i_{H},\qquad&&
\varepsilon_{I(\Phi_{H})}\coloneqq \varepsilon_{H}\circ i_{H}.&
\end{alignat*}
\end{Theorem}
\begin{proof}Let us start proving that $\bigl(I(\Phi_{H}),\eta_{I(\Phi_{H})},\mu_{I(\Phi_{H})}\bigr)$ is an algebra in $\sf{C}$. The unit property follows by
\begin{align*}
\mu_{I(\Phi_{H})}\circ\bigl(\eta_{I(\Phi_{H})}\otimes I(\Phi_{H})\bigr)
={}&p_{H}\circ\overline{\mu}_{H}\circ((\Phi_{H}\circ \eta_{H})\otimes i_{H})\ \text{\footnotesize\textnormal{(by \eqref{iconp})}}\\
={}&p_{H}\circ\overline{\mu}_{H}\circ(\eta_{H}\otimes i_{H})\ \text{\footnotesize\textnormal{(by (vi) of Definition \ref{WTPH})}}\\
={}&p_{H}\circ i_{H}\ \text{\footnotesize\textnormal{(by \eqref{muetaid})}}\\
={}&{\rm id}_{I(\Phi_{H})}\ \text{\footnotesize\textnormal{(by \eqref{pconi})}},
\end{align*}
and, on the other hand,
\begin{align*}
\mu_{I(\Phi_{H})}\circ\bigl(I(\Phi_{H})\otimes \eta_{I(\Phi_{H})}\bigr)
={}&p_{H}\circ\overline{\mu}_{H}\circ(i_{H}\otimes(\Phi_{H}\circ\eta_{H}))\ \text{\footnotesize\textnormal{(by \eqref{iconp})}}\\
={}&p_{H}\circ\overline{\mu}_{H}\circ(i_{H}\otimes\eta_{H})\ \text{\footnotesize\textnormal{(by (vi) of Definition \ref{WTPH})}}\\
={}&p_{H}\circ \Phi_{H}\circ i_{H}\ \text{\footnotesize\textnormal{(by \eqref{cond1WTPHA})}}\\
={}&(p_{H}\circ i_{H})\circ(p_{H}\circ i_{H})\ \text{\footnotesize\textnormal{(by \eqref{iconp})}}\\
={}&{\rm id}_{I(\Phi_{H})}\ \text{\footnotesize\textnormal{(by \eqref{pconi})}}.
\end{align*}

Now, let us see that $\mu_{I(\Phi_{H})}$ is associative,
\begin{gather*}
\mu_{I(\Phi_{H})}\circ\bigl(\mu_{I(\Phi_{H})}\otimes I(\Phi_{H})\bigr)\\
\qquad{}= p_{H}\circ\overline{\mu}_{H}\circ((\Phi_{H}\circ\overline{\mu}_{H}\circ(i_{H}\otimes i_{H}))\otimes i_{H})\ \text{\footnotesize\textnormal{(by \eqref{iconp})}}\\
\qquad{}= p_{H}\circ\overline{\mu}_{H}\circ((\overline{\mu}_{H}\circ(i_{H}\otimes i_{H}))\otimes i_{H})\ \text{\footnotesize\textnormal{(by (iii) of Definition \ref{WTPH} and \eqref{Phiconi})}}\\
\qquad{}= p_{H}\circ\overline{\mu}_{H}\circ(i_{H}\otimes(\overline{\mu}_{H}\circ(i_{H}\otimes i_{H})))\ \text{\footnotesize\textnormal{(by associativity of $\overline{\mu}_{H}$)}}\\
\qquad{}= p_{H}\circ\overline{\mu}_{H}\circ(i_{H}\otimes(\overline{\mu}_{H}\circ(i_{H}\otimes (\Phi_{H}\circ i_{H}))))\ \text{\footnotesize\textnormal{(by \eqref{Phiconi})}}\\
\qquad{}= p_{H}\circ\overline{\mu}_{H}\circ(i_{H}\otimes(\Phi_{H}\circ \overline{\mu}_{H}\circ(i_{H}\otimes i_{H})))\ \text{\footnotesize\textnormal{(by (iii) of Definition \ref{WTPH})}}\\
\qquad{}= \mu_{I(\Phi_{H})}\circ\bigl(I(\Phi_{H})\otimes\mu_{I(\Phi_{H})}\bigr)\ \text{\footnotesize\textnormal{(by \eqref{iconp} and definition of $\mu_{I(\Phi_{H})}$)}}.
\end{gather*}

In order to prove that $\bigl(I(\Phi_{H}),\varepsilon_{I(\Phi_{H})},\delta_{I(\Phi_{H})}\bigr)$ is a coalgebra in $\sf{C}$, note that the counit property is straightforward thanks to (ii.2) of Definition \ref{WTPH}. So, we will only detail that $\delta_{I(\Phi_{H})}$ is coassociative. Indeed,
\begin{gather*}
 \bigl(\delta_{I(\Phi_{H})}\otimes I(\Phi_{H})\bigr)\circ \delta_{I(\Phi_{H})}\\
 \qquad{}= (((p_{H}\otimes p_{H})\circ\delta_{H}\circ\Phi_{H})\otimes p_{H})\circ\delta_{H}\circ i_{H}\ \text{\footnotesize\textnormal{(by \eqref{iconp})}}\\
 \qquad{}= ((p_{H}\circ \Phi_{H})\otimes(p_{H}\circ \Phi_{H})\otimes p_{H})\circ(\delta_{H}\otimes H)\circ\delta_{H}\circ i_{H}\ \text{\footnotesize\textnormal{(by (ii.1) of Definition \ref{WTPH})}}\\
 \qquad{}= ((p_{H}\circ\Phi_{H})\otimes (p_{H}\circ \Phi_{H})\otimes p_{H})\circ(H\otimes \delta_{H})\circ\delta_{H}\circ i_{H}\ \text{\footnotesize\textnormal{(by coassociativity of $\delta_{H}$)}}\\
 \qquad{}= (p_{H}\otimes(p_{H}\circ\Phi_{H})\otimes(p_{H}\circ \Phi_{H}))\circ(H\otimes\delta_{H})\circ\delta_{H}\circ i_{H}\ \text{\footnotesize\textnormal{(by \eqref{Phiconi})}}\\
 \qquad{}= (p_{H}\otimes((p_{H}\otimes p_{H})\circ\delta_{H}\circ\Phi_{H}))\circ\delta_{H}\circ i_{H}\ \text{\footnotesize\textnormal{(by (ii.1) of Definition \ref{WTPH})}}\\
 \qquad{}= \bigl(I(\Phi_{H})\otimes\delta_{I(\Phi_{H})}\bigr)\circ\delta_{I(\Phi_{H})}\ \text{\footnotesize\textnormal{(by \eqref{iconp})}}.
\end{gather*}

Moreover, note that $\mu_{I(\Phi_{H})}$ and $\eta_{I(\Phi_{H})}$ are coalgebra morphisms. On the one hand, it is straightforward to compute that
\begin{gather*}
\begin{split}
& \varepsilon_{I(\Phi_{H})}\circ\mu_{I(\Phi_{H})}=\mu_{K}\circ(\varepsilon_{I(\Phi_{H})}\otimes\varepsilon_{I(\Phi_{H})})\overset{(\star)}{=}\varepsilon_{I(\Phi_{H})} \otimes\varepsilon_{I(\Phi_{H})},\\
& \delta_{I(\Phi_{H})}\circ\eta_{I(\Phi_{H})}=(\eta_{I(\Phi_{H})}\otimes \eta_{I(\Phi_{H})})\circ\delta_{K}\overset{(\star)}{=}\eta_{I(\Phi_{H})}\otimes \eta_{I(\Phi_{H})},
\end{split}
\end{gather*}
where the equalities indicated with ($\star$) follow by the fact that $\delta_{K}=\mu_{K}={\rm id}_{K}$ in a strict braided monoidal setting, i.e., the algebra and coalgebra structure over $K$ is the trivial one, and also
\begin{align*}
\varepsilon_{I(\Phi_{H})}\circ\eta_{I(\Phi_{H})}=\varepsilon_{H}\circ i_{H}\circ p_{H}\circ \eta_{H}=\varepsilon_{H}\circ\Phi_{H}\circ\eta_{H}=\varepsilon_{H}\circ\eta_{H}={\rm id}_{K}
\end{align*}
by \eqref{iconp} and the condition of coalgebra morphism for $\Phi_{H}$.

On the other hand,
\begin{gather*}
 \delta_{I(\Phi_{H})}\circ\mu_{I(\Phi_{H})}\\
 \qquad{}= (p_{H}\otimes p_{H})\circ\delta_{H}\circ\Phi_{H}\circ\overline{\mu}_{H}\circ (i_{H}\otimes i_{H})\ \text{\footnotesize\textnormal{(by \eqref{iconp})}}\\
 \qquad{}= (p_{H}\otimes p_{H})\circ\delta_{H}\circ\overline{\mu}_{H}\circ(i_{H}\otimes i_{H})\ \text{\footnotesize\textnormal{(by (ii.1) of Definition \ref{WTPH} and \eqref{Phiconi})}}\\
 \qquad{}= ((p_{H}\circ\overline{\mu}_{H})\otimes(p_{H}\circ\overline{\mu}_{H}))\circ(H\otimes c_{H,H}\otimes H)\circ((\delta_{H}\circ i_{H})\otimes(\delta_{H}\circ i_{H}))\\
 \qquad\quad{}\ \text{\footnotesize\textnormal{(by the condition of coalgebra morphism for $\overline{\mu}_{H}$)}}\\
 \qquad{}= ((p_{H}\circ\overline{\mu}_{H})\otimes(p_{H}\circ\overline{\mu}_{H}))\circ(\Phi_{H}\otimes (c_{H,H}\circ(\Phi_{H}\otimes \Phi_{H}))\otimes \Phi_{H})\\
 \qquad\quad{} \circ((\delta_{H}\circ i_{H})\otimes(\delta_{H}\circ i_{H}))\ \text{\footnotesize\textnormal{(by \eqref{Phiconi} and (ii.1) of Definition \ref{WTPH})}}\\
 \qquad{}= (\mu_{I(\Phi_{H})}\otimes\mu_{I(\Phi_{H})})\circ\bigl(I(\Phi_{H})\otimes c_{I(\Phi_{H}),I(\Phi_{H})}\otimes I(\Phi_{H})\bigr)\circ\bigl(\delta_{I(\Phi_{H})}\otimes\delta_{I(\Phi_{H})}\bigr)\\
 \qquad\quad{}\ \text{\footnotesize\textnormal{(by naturality of $c$ and \eqref{iconp})}}.\tag*{\qed}
\end{gather*}
\renewcommand{\qed}{}
\end{proof}

\begin{Remark}
Note that
\[
I(\Phi_{H})=\big(I(\Phi_{H}),\varepsilon_{I(\Phi_{H})},\delta_{I(\Phi_{H})}\big)
\]
is always a coalgebra in {\sf C} when $(H,m_{H},\Phi_{H})$ is an object in ${\sf wt}\textnormal{-}\sf{Post}\textnormal{-}\sf{Hopf}$ such that $\Phi_{H}\circ\eta_{H}=\eta_{H}$, which is a weaker condition than the one imposed in the previous theorem.
\end{Remark}
\begin{Corollary}\label{halgIphi}
Let $(H,m_{H},\Phi_{H})$ be an object in $\sf{coc}\textnormal{-}\sf{t}\textnormal{-}\sf{Post}\textnormal{-}\sf{Hopf}$ such that $\tilde{\beta}_{H}$ is a coalgebra morphism. $I(\Phi_{H})$ is a cocommutative Hopf algebra in $\sf{C}$ with antipode
\begin{equation}\label{deflambdaI}
\lambda_{I(\Phi_{H})}\coloneqq p_{H}\circ S_{H}\circ i_{H},
\end{equation}
where $S_{H}$ is defined as in \eqref{SHcocom}.
\end{Corollary}
\begin{proof}
We have to prove that $\lambda_{I(\Phi_{H})}$ satisfies \eqref{antipode}. On the one side,
\begin{align*}
{\rm id}_{I(\Phi_{H})}\ast \lambda_{I(\Phi_{H})}
={}&p_{H}\circ\overline{\mu}_{H}\circ(\Phi_{H}\otimes(\Phi_{H}\circ S_{H}\circ \Phi_{H}))\circ\delta_{H}\circ i_{H}\ \text{\footnotesize\textnormal{(by \eqref{deflambdaI} and \eqref{iconp})}}\\
={}&p_{H}\circ \overline{\mu}_{H}\circ(H\otimes S_{H})\circ\delta_{H}\circ i_{H}\ \text{\footnotesize\textnormal{(by (ii.1) and (iii) of Definition \ref{WTPH} and \eqref{Phiconi})}}\\
={}&\eta_{I(\Phi_{H})}\circ\varepsilon_{I(\Phi_{H})}\ \text{\footnotesize\textnormal{(by \eqref{convolution one side})}}.
\end{align*}

Note that
\begin{gather}
\overline{\mu}_{H}\circ(\Phi_{H}\otimes H)\nonumber\\
\qquad{} =\mu_{H}\circ((\Phi_{H}\circ \Phi_{H})\otimes(m_{H}\circ(\Phi_{H}\otimes H)))\circ(\delta_{H}\otimes H)\ \text{\footnotesize\textnormal{(by (ii.1) of Definition \ref{WTPH})}}\nonumber\\
\qquad{} =\mu_{H}\circ(\Phi_{H}\otimes m_{H})\circ(\delta_{H}\otimes H)\ \text{\footnotesize\textnormal{(by the idempotent character of $\Phi_{H}$ and \eqref{mconPHI})}}\nonumber\\
\qquad{} =\overline{\mu}_{H}\ \text{\footnotesize\textnormal{(by definition of $\overline{\mu}_{H}$)}}.\label{barmuPhi}
\end{gather}

Therefore, on the other side,
\begin{gather*}
\lambda_{I(\Phi_{H})}\ast {\rm id}_{I(\Phi_{H})}\\
 =p_{H}\circ\overline{\mu}_{H}\circ((\Phi_{H}\circ S_{H}\circ \Phi_{H})\otimes\Phi_{H})\circ\delta_{H}\circ i_{H}\ \text{\footnotesize\textnormal{(by \eqref{deflambdaI} and \eqref{iconp})}}\\
 =p_{H}\circ\overline{\mu}_{H}\circ((\Phi_{H}\circ S_{H}\circ \Phi_{H})\otimes(\Phi_{H}\circ S_{H}\circ S_{H}\circ \Phi_{H}))\circ\delta_{H}\circ i_{H}\ \text{\footnotesize\textnormal{(by \eqref{impliSH2=PHI})}}\\
 =p_{H}\circ\overline{\mu}_{H}\circ((\Phi_{H}\circ S_{H})\otimes (S_{H}\circ S_{H}))\circ\delta_{H}\circ i_{H}\ \text{\footnotesize\textnormal{(by (ii.1) and (iii) of Definition \ref{WTPH} and \eqref{Phiconi})}}\\
 =p_{H}\circ\overline{\mu}_{H}\circ(S_{H}\otimes(S_{H}\circ S_{H}))\circ\delta_{H}\circ i_{H}\ \text{\footnotesize\textnormal{(by \eqref{barmuPhi})}}\\
 =p_{H}\circ\overline{\mu}_{H}\circ(H\otimes S_{H})\circ\delta_{H}\circ S_{H}\circ i_{H}\ \text{\footnotesize\textnormal{(by the condition of coalgebra morphism for $S_{H}$)}}\\
 =p_{H}\circ\eta_{H}\circ\varepsilon_{H}\circ S_{H}\circ i_{H}\ \text{\footnotesize\textnormal{(by \eqref{convolution one side})}}\\
 =\eta_{I(\Phi_{H})}\circ\varepsilon_{I(\Phi_{H})}\ \text{\footnotesize\textnormal{(by the condition of coalgebra morphism for $S_{H}$)}}.\tag*{\qed}
\end{gather*}
\renewcommand{\qed}{}
\end{proof}

\section{Twisted relative Rota--Baxter operators and Hopf trusses}\label{s3}
In \cite{LST}, Y.~Li, Y.~Sheng and R.~Tang introduce the notion of relative Rota--Baxter operators for Hopf algebras as a generalisation of Rota--Baxter operators introduced by M. Goncharov in~\cite{Goncharov} for cocommutative Hopf algebras. Despite both notions were introduced in the category of vector spaces over a field $\mathbb{F}$, the same definitions can be used when we work in a braided monoidal framework because there is no influence of the braiding used. So, if $H$ and $B$ are Hopf algebras in ${\sf C}$ such that $(H,\phi_{H})$ is a left $B$-module algebra-coalgebra, a relative Rota--Baxter operator is a coalgebra morphism $T\colon H\rightarrow B$ satisfying the equality:
\begin{equation}\label{rRb}
\mu_{B}\circ(T\otimes T)=T\circ\mu_{H}\circ(H\otimes(\phi_{H}\circ (T\otimes H)))\circ(\delta_{H}\otimes H),
\end{equation}
i.e.,
\[
\includegraphics[scale=3]{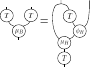}
\]

Note that every Rota--Baxter operator $B\colon H\rightarrow H$ in Goncharov's sense is a relative Rota--Baxter operator considering $\phi_{H}=\varphi_{H}^{{\rm ad}}$, the adjoint action, in the cocommutative setting. In~\cite{LST}, relative Rota--Baxter operators are used as a tool for finding new examples of post-Hopf algebras (and, as a consequence, of Hopf braces) and Y. Li et al. also construct a correspondence between the categories of relative Rota--Baxter operators and post-Hopf algebras that give rise to an adjoint pair between the respective functors (see \cite[Theorem 3.3]{LST}).

In this section, we are going to introduce the notion of weak twisted relative Rota--Baxter operators, whose main difference from the usual ones is that condition \eqref{rRb} is modified through an endomorphism $\Psi_{H}$ of $H$. In what follows we are going to prove that, while relative Rota--Baxter operators are in correspondence with post-Hopf algebras and Hopf braces, these new objects give rise to a correspondence with Hopf trusses, obtaining generalisations to this new context of the results above-mentioned. Moreover, if we consider the subcategory of weak twisted relative Rota--Baxter operators such that $T$ is an isomorphism, we will also prove that this subcategory is equivalent to the category of Hopf trusses.
\begin{Definition}\label{WTRB}
Let $H=(H,\eta_{H},\mu_{H},\varepsilon_{H},\delta_{H},\lambda_{H})$ be a Hopf algebra and let $B=(B,\mu_{B},\varepsilon_{B},\delta_{B})$ be a non-unital bialgebra in $\sf{C}$. Suppose that there exists a morphism $\varphi_{H}\colon B\otimes H\rightarrow H$ such that $(H,\varphi_{H})$ is a non-unital left $B$-module algebra-coalgebra. We will say that a~coalgebra morphism $T\colon H\rightarrow B$ is a weak twisted relative Rota--Baxter operator if there exists ${\Psi_{H}\colon H\rightarrow H}$ a~coalgebra morphism, called the cocycle of $T$, such that the following conditions hold:
\begin{itemize}\itemsep=0pt
\item[(i)] $\mu_{B}\circ (T\otimes T)=T\circ \mu_{H}\circ (\Psi_{H}\otimes(\varphi_{H}\circ(T\otimes H)))\circ(\delta_{H}\otimes H)$,
\item[(ii)] $\Psi_{H}\circ\mu_{H}\circ (\Psi_{H}\otimes(\varphi_{H}\circ(T\otimes H)))\circ(\delta_{H}\otimes H)=\mu_{H}\circ (\Psi_{H}\otimes(\varphi_{H}\circ(T\otimes H)))\circ(\delta_{H}\otimes \Psi_{H})$.
\end{itemize}

If it also holds that
\begin{itemize}\itemsep=0pt
\item[(iii)] $\Psi_{H}\circ \eta_{H}=\eta_{H}$, i.e., the cocycle $\Psi_{H}$ preserves the unit,
\end{itemize}
then $T\colon H\rightarrow B$ is said to be a twisted relative Rota--Baxter operator.

In what follows we will denote (weak) twisted relative Rota--Baxter operators by
\[
\Biggl(T\begin{array}{c} H\\\downarrow\\B\end{array},\varphi_{H},\Psi_{H}\Biggr).
\]
\end{Definition}
We are going to define by $\mathfrak{m}_{H}\coloneqq \varphi_{H}\circ (T\otimes H)$, i.e.,
\[
\includegraphics[scale=3]{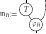}
\]
\noindent Therefore, conditions (i) and (ii) of previous definition are equivalent to
\begin{itemize}\itemsep=0pt
\item[(i)] $\mu_{B}\circ (T\otimes T)=T\circ\mu_{H}\circ (\Psi_{H}\otimes\mathfrak{m}_{H})\circ(\delta_{H}\otimes H)$, i.e.,
\[
\includegraphics[scale=3]{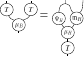}
\]

\item[(ii)] $\Psi_{H}\circ \mu_{H}\circ (\Psi_{H}\otimes\mathfrak{m}_{H})\circ(\delta_{H}\otimes H)=\mu_{H}\circ (\Psi_{H}\otimes\mathfrak{m}_{H})\circ(\delta_{H}\otimes \Psi_{H})$, i.e.,
\[
\includegraphics[scale=3]{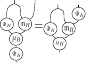}
\]
\end{itemize}
\begin{Definition}
Let
\[
\Biggl(T\begin{array}{c}H\\\downarrow\\B\end{array},\varphi_{H},\Psi_{H}\Biggr) \qquad \text{and} \qquad \Biggl(T'\begin{array}{c}H'\\\downarrow\\B'\end{array},\varphi_{H'},\Psi_{H'}\Biggr)
\] be weak twisted relative Rota--Baxter operators. We will say that a pair
\[
(f,g)\colon \ \Biggl(T\begin{array}{c}H\\\downarrow\\B\end{array},\varphi_{H},\Psi_{H}\Biggr)\rightarrow \Biggl(T'\begin{array}{c}H'\\\downarrow\\B'\end{array},\varphi_{H'},\Psi_{H'}\Biggr),
\]
where $f\colon H\rightarrow H'$ is a Hopf algebra morphism and $g\colon B\rightarrow B'$ is a morphism of non-unital bialgebras, is a morphism of weak twisted relative Rota--Baxter operators if the following conditions hold:
\begin{gather}
\label{mor0twrRb}
T'\circ f=g\circ T,\\
\label{mor1twrRb}
f\circ \Psi_{H}=\Psi_{H'}\circ f,\\
\label{mor2.0twrRb}
f\circ \varphi_{H}=\varphi_{H'}\circ(g\otimes f).
\end{gather}
\end{Definition}

So, weak twisted relative Rota--Baxter operators give rise to a category that we will denote by~${\sf wtr}$-${\sf RB}$, for which twisted relative Rota--Baxter operators constitute a full subcategory denoted by ${\sf tr}$-${\sf RB}$.
\begin{Remark}
Consider
\[
\Biggl(T\begin{array}{c}H\\\downarrow\\B\end{array},\varphi_{H},\Psi_{H}\Biggr)
\]
 a weak twisted relative Rota--Baxter operator. Due to~$T$ being a coalgebra morphism and $(H,\varphi_{H})$ a non-unital left $B$-module algebra-coalgebra, it is straightforward to prove that the following equalities hold:
\begin{gather}\label{etamH}
\mathfrak{m}_{H}\circ(H\otimes\eta_{H})=\varepsilon_{H}\otimes\eta_{H},\\
\label{mumH}
\mathfrak{m}_{H}\circ(H\otimes\mu_{H})=\mu_{H}\circ(\mathfrak{m}_{H}\otimes\mathfrak{m}_{H})\circ(H\otimes c_{H,H}\otimes H)\circ(\delta_{H}\otimes H\otimes H),\\
%\label{epsmH}
\varepsilon_{H}\circ\mathfrak{m}_{H}=\varepsilon_{H}\otimes\varepsilon_{H},\nonumber\\
\label{deltamH}
\delta_{H}\circ \mathfrak{m}_{H}=(\mathfrak{m}_{H}\otimes\mathfrak{m}_{H})\circ(H\otimes c_{H,H}\otimes H)\circ(\delta_{H}\otimes\delta_{H}).
\end{gather}

Moreover, the equality
\begin{equation}\label{mHtildemuH}
\mathfrak{m}_{H}\circ((\mu_{H}\circ(\Psi_{H}\otimes \mathfrak{m}_{H})\circ(\delta_{H}\otimes H))\otimes H)=\mathfrak{m}_{H}\circ(H\otimes\mathfrak{m}_{H}),
\end{equation}
i.e.,
\[
\includegraphics[scale=3]{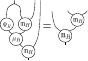}
\]
\noindent also holds. Indeed,
\begin{align*}
&\mathfrak{m}_{H}\circ((\mu_{H}\circ(\Psi_{H}\otimes \mathfrak{m}_{H})\circ(\delta_{H}\otimes H))\otimes H)\\
&\qquad{}=\varphi_{H}\circ((\mu_{B}\circ(T\otimes T))\otimes H)\ \text{\footnotesize\textnormal{(by (i) of Definition \ref{WTRB})}}\\
&\qquad{}=\mathfrak{m}_{H}\circ(H\otimes\mathfrak{m}_{H})\ \text{\footnotesize\textnormal{(by the $B$-module conditions)}}.
\end{align*}
\end{Remark}
\begin{Remark}
Note that if $(f,g)$ is a morphism between the weak twisted relative Rota--Baxter operators
\[
\Biggl(T\begin{array}{c}H\\\downarrow\\B\end{array},\varphi_{H},\Psi_{H}\Biggl) \qquad \text{and}\qquad \Biggl(T'\begin{array}{c}H'\\\downarrow\\B'\end{array},\varphi_{H'},\Psi_{H'}\Biggr),
\]
then \begin{equation}\label{mor2twrRb}
f\circ\mathfrak{m}_{H}=\mathfrak{m}_{H'}\circ(f\otimes f).
\end{equation}
Indeed, by \eqref{mor2.0twrRb} and \eqref{mor0twrRb}, we have
\[f\circ\mathfrak{m}_{H}=\varphi_{H'}\circ((g\circ T)\otimes f)=\varphi_{H'}\circ\bigl(\bigl(T'\circ f\bigr)\otimes f\bigr)=\mathfrak{m}_{H'}\circ(f\otimes f).
\]
\end{Remark}
\begin{Theorem}%\label{etaTlem}
Let
\[
\Biggl(T\begin{array}{c}H\\\downarrow\\B\end{array},\varphi_{H},\Psi_{H}\Biggl)
\]
 be a twisted relative Rota--Baxter operator such that $B$ is a bialgebra in ${\sf C}$. It is satisfied that
\begin{equation}\label{etaT}
T\circ\eta_{H}=\eta_{B}.
\end{equation}
\end{Theorem}
\begin{proof}
By (i) of Definition \ref{WTRB}, we obtain that
\begin{align}
&\mu_{B}\circ((T\circ\eta_{H})\otimes (T\circ\eta_{H}))\nonumber\\
&\qquad{}=T\circ\mu_{H}\circ(\Psi_{H}\otimes\mathfrak{m}_{H})\circ((\delta_{H}\circ\eta_{H})\otimes \eta_{H})\ \text{\footnotesize\textnormal{(by (i) of Definition \ref{WTRB})}}\nonumber\\
&\qquad{}=T\circ\mu_{H}\circ(\Psi_{H}\otimes(\eta_{H}\circ\varepsilon_{H}))\circ\delta_{H}\circ\eta_{H}\ \text{\footnotesize\textnormal{(by \eqref{etamH})}}\nonumber\\
&\qquad{}=T\circ\eta_{H}\ \text{\footnotesize\textnormal{(by the (co)unit properties and (iii) of Definition \ref{WTRB})}}.\label{etaTproof1}
\end{align}

Therefore,
\begin{align*}
\eta_{B}&{}=\eta_{B}\circ\varepsilon_{B}\circ T\circ\eta_{H}\ \text{\footnotesize\textnormal{(by (co)unit properties and the condition of coalgebra morphism for $T$)}}\\
&{}=\mu_{B}\circ(\lambda_{B}\otimes B)\circ\delta_{B}\circ T\circ\eta_{H}\ \text{\footnotesize\textnormal{(by \eqref{antipode})}}\\
&{}=\mu_{B}\circ((\lambda_{B}\circ T\circ\eta_{H})\otimes (T\circ\eta_{H}))\ \text{\footnotesize\textnormal{(by the condition of coalgebra morphism for $T$ and $\eta_{H}$)}}\\
&{}=\mu_{B}\circ((\lambda_{B}\circ T\circ\eta_{H})\otimes(\mu_{B}\circ ((T\circ\eta_{H})\otimes(T\circ\eta_{H}))))\ \text{\footnotesize\textnormal{(by \eqref{etaTproof1})}}\\
&{}=\mu_{B}\circ((\mu_{B}\circ(\lambda_{B}\otimes B)\circ\delta_{B}\circ T\circ\eta_{H})\otimes(T\circ\eta_{H}))\\
&\quad{}\ \text{\footnotesize\textnormal{(by associativity of $\mu_{B}$ and the condition of coalgebra morphism for $T$ and $\eta_{H}$)}}\\
&{}=\mu_{B}\circ((\eta_{B}\circ\varepsilon_{B}\circ T\circ\eta_{H})\otimes(T\circ\eta_{H}))\ \text{\footnotesize\textnormal{(by \eqref{antipode})}}\\
&{}=T\circ\eta_{H}\ \text{\footnotesize\textnormal{(by the condition of coalgebra morphism for $T$ and (co)unit properties)}}.\tag*{\qed}
\end{align*}
\renewcommand{\qed}{}
\end{proof}

\begin{Corollary}
Let
\[
\Biggl(T\begin{array}{c}H\\\downarrow\\B\end{array},\varphi_{H},\Psi_{H}\Biggr)
\]
 be a twisted relative Rota--Baxter operator such that $B$ is a bialgebra in ${\sf C}$ and $(H,\varphi_{H})$ is a left $B$-module algebra-coalgebra. The equality
\begin{equation}\label{frakmHeta}
\mathfrak{m}_{H}\circ(\eta_{H}\otimes H)={\rm id}_{H}
\end{equation}
holds.
\end{Corollary}
\begin{proof}
Using \eqref{etaT} and the condition of left $B$-module, we have that
\begin{align*}
&\mathfrak{m}_{H}\circ(\eta_{H}\otimes H)=\varphi_{H}\circ(\eta_{B}\otimes H)={\rm id}_{H}.\tag*{\qed}
\end{align*}
\renewcommand{\qed}{}
\end{proof}

\begin{Theorem}
Let
\[
\Biggl(T\begin{array}{c}H\\\downarrow\\B\end{array},\varphi_{H},\Psi_{H}\Biggr)
\]
 be a weak twisted relative Rota--Baxter operator such that
\begin{gather}\label{classcocomfrakmH}
(\mathfrak{m}_{H}\otimes H)\circ(H\otimes c_{H,H})\circ ((c_{H,H}\circ\delta_{H})\otimes H )=(\mathfrak{m}_{H}\otimes H)\circ(H\otimes c_{H,H})\circ(\delta_{H}\otimes H),\!\!\!
\end{gather}
i.e.,
\[
\includegraphics[scale=3]{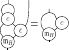}
\]
holds. Then, $\widetilde{H}=(H,\widetilde{\mu}_{H},\varepsilon_{H},\delta_{H})$, where \[\widetilde{\mu}_{H}\coloneqq \mu_{H}\circ (\Psi_{H}\otimes\mathfrak{m}_{H})\circ(\delta_{H}\otimes H),\]
i.e.,
\[
\includegraphics[scale=3]{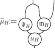}
\]
is a non-unital bialgebra in ${\sf C}$ that satisfies
\begin{equation}\label{etadrchmutilde}
\widetilde{\mu}_{H}\circ(H\otimes \eta_{H})=\Psi_{H}.
\end{equation}

Moreover, if
\[
\Biggl(T\begin{array}{c}H\\\downarrow\\B\end{array},\varphi_{H},\Psi_{H}\Biggr)
\]
 is a twisted relative Rota--Baxter operator such that $B$ is a~bialgebra and $(H,\varphi_{H})$ is a left $B$-module algebra-coalgebra, then
\begin{equation}\label{etaizqmutilde}
\widetilde{\mu}_{H}\circ(\eta_{H}\otimes H)={\rm id}_{H}.
\end{equation}
\end{Theorem}
\begin{proof}
Let us start computing the associativity of $\widetilde{\mu}_{H}$:
\begin{gather*}
\widetilde{\mu}_{H}\circ(\widetilde{\mu}_{H}\otimes H)\\
= \mu_{H}\circ(\Psi_{H}\otimes \mathfrak{m}_{H})\circ(((\mu_{H}\otimes\mu_{H})\circ(H\otimes c_{H,H}\otimes H)\circ((\delta_{H}\circ\Psi_{H})\otimes(\delta_{H}\circ \mathfrak{m}_{H})))\otimes H)\\
\quad{}\circ(\delta_{H}\otimes H\otimes H)\ \text{\footnotesize\textnormal{(by the condition of coalgebra morphism for $\mu_{H}$)}}\\
= \mu_{H}\circ(\Psi_{H}\otimes \mathfrak{m}_{H})\circ (((\mu_{H}\otimes \mu_{H})\circ(H\otimes c_{H,H}\otimes H))\otimes H)\\
\quad{}\circ (((\Psi_{H}\otimes \Psi_{H})\circ\delta_{H})\otimes((\mathfrak{m}_{H}\otimes \mathfrak{m}_{H})\circ(H\otimes c_{H,H}\otimes H)\circ(\delta_{H}\otimes\delta_{H}))\otimes H)\\
\quad{}\circ(\delta_{H}\otimes H\otimes H)\ \text{\footnotesize\textnormal{(by the condition of coalgebra morphism for $\Psi_{H}$ and \eqref{deltamH})}}\\
= \mu_{H}\circ ((\Psi_{H}\circ\mu_{H}\circ(\Psi_{H}\otimes H))\otimes(\mathfrak{m}_{H}\circ((\mu_{H}\circ(\Psi_{H}\otimes \mathfrak{m}_{H}))\otimes H)))\\
\quad{}\circ(H\otimes ((\mathfrak{m}_{H}\otimes H)\circ (H\otimes c_{H,H})\circ((c_{H,H}\circ\delta_{H})\otimes H))\otimes H\otimes H\otimes H)\\
\quad{}\circ(H\otimes((H\otimes c_{H,H}\otimes H)\circ(\delta_{H}\otimes\delta_{H}))\otimes H)\circ(\delta_{H}\otimes H\otimes H)\\
\quad{}\ \text{\footnotesize\textnormal{(by naturality of $c$ and coassociativity of $\delta_{H}$)}}\\
= \mu_{H}\\
\quad{}\circ((\Psi_{H}\circ\mu_{H}\circ (\Psi_{H}\otimes \mathfrak{m}_{H})\circ(\delta_{H}\otimes H))\otimes(\mathfrak{m}_{H}\circ((\mu_{H}\circ(\Psi_{H}\otimes \mathfrak{m}_{H})\circ(\delta_{H}\otimes H))\otimes H)))\\
\quad{} \circ(((H\otimes c_{H,H}\otimes H)\circ(\delta_{H}\otimes\delta_{H}))\otimes H)\ \text{\footnotesize\textnormal{(by \eqref{classcocomfrakmH}, coassociativity of $\delta_{H}$ and naturality of $c$)}}\\
= \mu_{H}\circ((\mu_{H}\circ (\Psi_{H}\otimes \mathfrak{m}_{H})\circ(\delta_{H}\otimes \Psi_{H}))\otimes(\mathfrak{m}_{H}\circ(H\otimes \mathfrak{m}_{H})))\\
\quad{}\circ(((H\otimes c_{H,H}\otimes H)\circ(\delta_{H}\otimes\delta_{H}))\otimes H)\ \text{\footnotesize\textnormal{(by (ii) of Definition \ref{WTRB} and \eqref{mHtildemuH})}}\\
= \mu_{H}\circ (\Psi_{H}\otimes (\mu_{H}\circ(\mathfrak{m}_{H}\otimes \mathfrak{m}_{H})\circ(H\otimes c_{H,H}\otimes H)\circ(\delta_{H}\otimes H\otimes H)))\\
\quad{}\circ(\delta_{H}\otimes ((\Psi_{H}\otimes \mathfrak{m}_{H})\circ(\delta_{H}\otimes H)))\\
\quad{}\ \text{\footnotesize\textnormal{(by naturality of $c$, coassociativity of $\delta_{H}$ and associativity of $\mu_{H}$)}}\\
= \widetilde{\mu}_{H}\circ(H\otimes\widetilde{\mu}_{H})\ \text{\footnotesize\textnormal{(by \eqref{mumH})}}.
\end{gather*}

Moreover, note that $\widetilde{\mu}_{H}$ is a coalgebra morphism. On the one hand, it is straightforward to prove that $\varepsilon_{H}\circ\widetilde{\mu}_{H}=\varepsilon_{H}\otimes\varepsilon_{H}$ and, on the other hand,
\begin{gather*}
\delta_{H}\circ\widetilde{\mu}_{H}\\
\qquad{}=(\mu_{H}\otimes\mu_{H})\circ(H\otimes c_{H,H}\otimes H)\circ((\delta_{H}\circ\Psi_{H})\otimes(\delta_{H}\circ\mathfrak{m}_{H}))\circ(\delta_{H}\otimes H)\\
\qquad\quad{}\ \text{\footnotesize\textnormal{(by the condition of coalgebra morphism for $\mu_{H}$)}}\\
\qquad{}=(\mu_{H}\otimes\mu_{H})\circ(H\otimes c_{H,H}\otimes H)\\
\qquad\quad{}\circ(((\Psi_{H}\otimes\Psi_{H})\circ\delta_{H})\otimes((\mathfrak{m}_{H}\otimes\mathfrak{m}_{H})\circ(H\otimes c_{H,H}\otimes H)
\circ(\delta_{H}\otimes\delta_{H})))\circ(\delta_{H}\otimes H)\\
\qquad\quad{}\ \text{\footnotesize\textnormal{(by the condition of coalgebra morphism for $\Psi_{H}$ and \eqref{deltamH})}}\\
\qquad{}=((\mu_{H}\circ(\Psi_{H}\otimes H))\otimes(\mu_{H}\circ(\Psi_{H}\otimes\mathfrak{m}_{H})))\\
\qquad\quad{}\circ(H\otimes((\mathfrak{m}_{H}\otimes H)\circ(H\otimes c_{H,H})\circ((c_{H,H}\circ\delta_{H})\otimes H))\otimes H\otimes H)\\
\qquad\quad{}\circ(H\otimes H\otimes c_{H,H}\otimes H)\circ(((H\otimes\delta_{H})\circ\delta_{H})\otimes\delta_{H})\\
\qquad\quad{}\ \text{\footnotesize\textnormal{(by naturality of $c$ and coassociativity of $\delta_{H}$)}}\\
\qquad{}=((\mu_{H}\circ(\Psi_{H}\otimes H))\otimes(\mu_{H}\circ(\Psi_{H}\otimes\mathfrak{m}_{H})))\\
\qquad\quad{}\circ(H\otimes((\mathfrak{m}_{H}\otimes H)\circ(H\otimes c_{H,H})\circ(\delta_{H}\otimes H))
\otimes H\otimes H)\\
\qquad\quad{}\circ(H\otimes H\otimes c_{H,H}\otimes H)\circ(((H\otimes\delta_{H})\circ\delta_{H})\otimes\delta_{H})\ \text{\footnotesize\textnormal{(by \eqref{classcocommH})}}\\
\qquad{}=(\widetilde{\mu}_{H}\otimes\widetilde{\mu}_{H})\circ(H\otimes c_{H,H}\otimes H)\circ(\delta_{H}\otimes\delta_{H})\ \text{\footnotesize\textnormal{(by coassociativity of $\delta_{H}$ and naturality of $c$)}}.
\end{gather*}

To finish, let us see that the unit properties \eqref{etadrchmutilde} and \eqref{etaizqmutilde} hold. Indeed, on the one hand,
\[
\widetilde{\mu}_{H}\circ(H\otimes\eta_{H})=\Psi_{H}
\]
by \eqref{etamH} and, on the other hand, \eqref{etaizqmutilde} follows by
\begin{align*}
&\widetilde{\mu}_{H}\circ(\eta_{H}\otimes H)\\
&\qquad{}=\mu_{H}\circ((\Psi_{H}\circ\eta_{H})\otimes(\mathfrak{m}_{H}\circ(\eta_{H}\otimes H)))\ \text{\footnotesize\textnormal{(by the condition of coalgebra morphism for $\eta_{H}$)}}\\
&\qquad{}=\mu_{H}\circ(\eta_{H}\otimes H)\ \text{\footnotesize\textnormal{(by (iii) of Definition \ref{WTRB} and \eqref{frakmHeta})}}\\
&\qquad{}={\rm id}_{H}\ \text{\footnotesize\textnormal{(by the unit property)}}.\tag*{\qed}
\end{align*}
\renewcommand{\qed}{}
\end{proof}

\begin{Remark}
Note that, thanks to \eqref{mHtildemuH} and previous theorem, $(H,\mathfrak{m}_{H})$ is a non-unital left $\widetilde{H}$-module. Then, if ${\sf C}$ is symmetric, we can say that $(H,\mathfrak{m}_{H})$ is in the cocommutativity class of~\smash{$\widetilde{H}$} when \eqref{classcocomfrakmH} holds.
\end{Remark}

\begin{Theorem}\label{funOmega}
Let
\[
\Biggl(T\begin{array}{c}H\\\downarrow\\B\end{array},\varphi_{H},\Psi_{H}\Biggr)
\]
 be a weak twisted relative Rota--Baxter operator such that \eqref{classcocomfrakmH} holds. Then, the triple $\widetilde{\mathbb{H}}=\bigl(H,\widetilde{H},\Psi_{H}\bigr)$ is an object in $\sf{HTr}^{\star}$.
\end{Theorem}

\begin{proof}
At first, note that
\begin{equation}\label{GammaPsiH}
\Gamma_{H}^{\Psi_{H}}=\mathfrak{m}_{H}.
\end{equation}

Indeed,
\begin{align*}
\Gamma_{H}^{\Psi_{H}}
&{}=\mu_{H}\circ((\mu_{H}\circ((\lambda_{H}\circ\Psi_{H})\otimes\Psi_{H})\circ\delta_{H})\otimes\mathfrak{m}_{H})\circ(\delta_{H}\otimes H)\\
&\quad{}\, \text{\footnotesize\textnormal{(by associativity of $\mu_{H}$ and coassociativity of $\delta_{H}$)}}\\
&{}=\mu_{H}\circ((\mu_{H}\circ(\lambda_{H}\otimes H)\circ\delta_{H}\circ\Psi_{H})\otimes\mathfrak{m}_{H})\circ(\delta_{H}\otimes H)\\
&\quad{}\, \text{\footnotesize\textnormal{(by the condition of coalgebra morphism for $\Psi_{H}$)}}\\
&{}=\mu_{H}\circ((\eta_{H}\circ\varepsilon_{H}\circ\Psi_{H})\otimes\mathfrak{m}_{H})\circ(\delta_{H}\otimes H)\ \text{\footnotesize\textnormal{(by \eqref{antipode})}}\\
&{}=\mathfrak{m}_{H}\ \text{\footnotesize\textnormal{(by the condition of coalgebra morphism for $\Psi_{H}$ and (co)unit properties)}}.
\end{align*}

Thus, to conclude that $\widetilde{\mathbb{H}}$ is an object in $\sf{HTr}^{\star}$ it is enough to compute (iii) of Definition \ref{H-truss}. Indeed,
\begin{gather*}
\mu_{H}\circ\bigl(\widetilde{\mu}_{H}\otimes\Gamma_{H}^{\Psi_{H}}\bigr)\circ(H\otimes c_{H,H}\otimes H)\circ(\delta_{H}\otimes H\otimes H)\\
\quad{}=\mu_{H}\circ((\mu_{H}\circ(\Psi_{H}\otimes\mathfrak{m}_{H})\circ(\delta_{H}\otimes H))\otimes \mathfrak{m}_{H})\circ(H\otimes c_{H,H}\otimes H)\circ(\delta_{H}\otimes H\otimes H)\\
\qquad{}\;\text{\footnotesize\textnormal{(by \eqref{GammaPsiH})}}\\
\quad{}=\mu_{H}\circ(\Psi_{H}\otimes(\mu_{H}\circ(\mathfrak{m}_{H}\otimes\mathfrak{m}_{H})\circ(H\otimes c_{H,H}\otimes H)\circ(\delta_{H}\otimes H\otimes H)))\circ(\delta_{H}\otimes H\otimes H)\\
\qquad{}\; \text{\footnotesize\textnormal{(by associativity of $\mu_{H}$ and coassociativity of $\delta_{H}$)}}\\
\quad{}=\widetilde{\mu}_{H}\circ(H\otimes\mu_{H})\ \text{\footnotesize\textnormal{(by \eqref{mumH})}}.\tag*{\qed}
\end{gather*}
\renewcommand{\qed}{}
\end{proof}

The previous result can be interpreted as follows: If we denote by ${\sf wtr}$-${\sf RB}^{\star}$ to the full subcategory of ${\sf wtr}$-${\sf RB}$ satisfying the condition \eqref{classcocomfrakmH}, there exists a functor
$\Omega\colon {\sf wtr}\textnormal{-}{\sf RB}^{\star}\longrightarrow \sf{HTr}^{\star}$
defined on objects by
\[
\Omega\Biggl(\Biggl(T\begin{array}{c}H\\\downarrow\\B\end{array},\varphi_{H},\Psi_{H}\Biggr)\Biggr)=\widetilde{\mathbb{H}}
\]
and on morphisms by $\Omega((f,g))=f$. Note that $\Omega$ is well defined on morphisms because
\begin{gather*}
f\circ\widetilde{\mu}_{H}\\
\qquad{}=\mu_{H'}\circ((f\circ\Psi_{H})\otimes(f\circ\mathfrak{m}_{H}))\circ(\delta_{H}\otimes H)\ \text{\footnotesize\textnormal{(by the condition of algebra morphism for $f$)}}\\
\qquad{}=\mu_{H'}\circ((\Psi_{H'}\circ f)\otimes(\mathfrak{m}_{H'}\circ(f\otimes f)))\circ(\delta_{H}\otimes H)\ \text{\footnotesize\textnormal{(by \eqref{mor1twrRb} and \eqref{mor2twrRb})}}\\
\qquad{}=\mu_{H'}\circ(\Psi_{H'}\otimes\mathfrak{m}_{H'})\circ((\delta_{H'}\circ f)\otimes f)\ \text{\footnotesize\textnormal{(by the condition of coalgebra morphism for $f$)}}\\
\qquad{}=\widetilde{\mu}_{H'}\circ(f\otimes f).
\end{gather*}
\begin{Theorem}\label{funLambda}
Let $\mathbb{H}=(H_{1},H_{2},\sigma_{H})$ be an object in $\sf{HTr}^{\star}$. Then, the triple
\[
\Biggl({\rm id}_{H}\begin{array}{c}H_{1}\\\downarrow\\H_{2}\end{array},\Gamma_{H_{1}}^{\sigma_{H}},\sigma_{H}\Biggr)
\]
 is an object in ${\sf wtr}\textnormal{-}{\sf RB}^{\star}$. As a consequence, there exists a functor
$\Lambda\colon \sf{HTr}^{\star}\longrightarrow {\sf wtr}\textnormal{-}{\sf RB}^{\star}$
acting on objects by
\[
\Lambda(\mathbb{H})=\Biggl({\rm id}_{H}\begin{array}{c}H_{1}\\\downarrow\\H_{2}\end{array},\Gamma_{H_{1}}^{\sigma_{H}},\sigma_{H}\Biggr)
\]
and on morphisms by $\Lambda(f)=(f,f)$.
\end{Theorem}
\begin{proof}
Firstly, it is already known that \smash{$\bigl(H_{1},\Gamma_{H_{1}}^{\sigma_{H}}\bigr)$} is a non-unital left $H_{2}$-module algebra. Moreover, thanks to \eqref{classcocomGH}, \smash{$\Gamma_{H_{1}}^{\sigma_{H}}$} is a coalgebra morphism as we have just proved in Theorem~\ref{functorG}. As a consequence, \smash{$\bigl(H_{1},\Gamma_{H_{1}}^{\sigma_{H}}\bigr)$} is a non-unital left $H_{2}$-module algebra-coalgebra.

Moreover, ${\rm id}_{H}\colon H_{1}\rightarrow H_{2}$ is a coalgebra morphism because both have the same underlying coalgebra structure, and also $\sigma_{H}$ due to Hopf truss' axioms.

So, to show that $\Lambda$ is well defined on objects it is enough to see that conditions (i) and (ii) of Definition \ref{WTRB} hold. In this context, equation (i) of Definition \ref{WTRB} becomes
\[
\mu_{H}^{2}=\mu_{H}^{1}\circ\bigl(\sigma_{H}\otimes\Gamma_{H_{1}}^{\sigma_{H}}\bigr)\circ(\delta_{H}\otimes H)
\]
 that holds due to \eqref{mu2Htruss}. On the other hand, (ii) of Definition \ref{WTRB} follows by
\begin{align*}
\sigma_{H}\circ\mu_{H}^{1}\circ\bigl(\sigma_{H}\otimes\Gamma_{H_{1}}^{\sigma_{H}}\bigr)\circ(\delta_{H}\otimes H)
={}&\sigma_{H}\circ\mu_{H}^{2}\ \text{\footnotesize\textnormal{(by \eqref{mu2Htruss})}}\\
={}&\mu_{H}^{2}\circ(H\otimes\sigma_{H})\ \text{\footnotesize\textnormal{(by \eqref{cocycle truss 2})}}\\
={}&\mu_{H}^{1}\circ\bigl(\sigma_{H}\otimes\Gamma_{H_{1}}^{\sigma_{H}}\bigr)\circ(\delta_{H}\otimes\sigma_{H})\ \text{\footnotesize\textnormal{(by \eqref{mu2Htruss})}}.
\end{align*}

In addition, by \eqref{Cond mor truss} and \eqref{gammatrussf}, $\Lambda$ is well defined on morphisms too, what concludes the proof.
\end{proof}

\begin{Example}
Following \cite[Definition 3.5]{LW2}, given a Hopf algebra $D=(D,\eta_{D},\mu_{D},\varepsilon_{D},\delta_{D},\lambda_{D})$ and a Hopf algebra endomorphism $\phi\colon D\rightarrow D$, a $\phi$-twisted operator is a coalgebra morphism $\Upsilon \colon D\rightarrow D$ such that the equation
\begin{equation*}
	\mu_{D}\circ(\Upsilon\otimes\Upsilon)=\Upsilon\circ\mu_{D}\circ((\mu_{D}\circ(\Upsilon\otimes D))\otimes (\lambda_{D}\circ \phi\circ \Upsilon))\circ(D\otimes c_{D,D})\circ(\delta_{D}\otimes D),
\end{equation*}
i.e.,
\[
\includegraphics[scale=2.5]{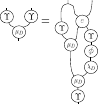}
\]
holds.

If $q:D\rightarrow D$ is a coalgebra morphism satisfying
\begin{equation}\label{condBidemp}
	\mu_{D}\circ(q\otimes q)=q\circ\mu_{D}\circ (q\otimes D),
\end{equation}
we have that $q$ is a $\phi_{q}$-twisted operator, where $\phi_{q}=\eta_{D}\circ \varepsilon_{D}$, and ${\mathbb D}_{q}=\bigl(D, D_{q}, \sigma_{D}^{q}\bigr)$ is a Hopf truss where $D_{q}$ is the non-unital bialgebra
\begin{gather*}
	D_{q}=\bigl(D,\mu_{D}^q=\mu_{D}\circ (q\otimes D),\varepsilon_{D},\delta_{D}\bigr),
\end{gather*}
and $\sigma_{D}^{q}=q$. Note that in this case $\Gamma_{D}^q=\varepsilon_{D}\otimes D$. As a consequence ${\mathbb D}_{q}$ is an object in $\sf{HTr}^{\star}$. Then, by Theorem \ref{funOmega}, the triple
\[
\Biggl({\rm id}_{D}\begin{array}{c}D\\\downarrow\\D_{q}\end{array},\Gamma_{D}^q=\varepsilon_{D}\otimes D,\sigma_{D}^q=q\Biggr)
\]
is a weak twisted relative Rota--Baxter operator satisfying condition \eqref{classcocomfrakmH}.

In this setting, a family of morphisms $q\colon D\rightarrow D$ satisfying \eqref{condBidemp} is made up of idempotent Hopf algebra endomorphisms of $D$. So, we can construct examples of Hopf trusses in $\sf{HTr}^{\star}$ and weak twisted relative Rota--Baxter operators satisfying \eqref{classcocomfrakmH}, working with idempotent Hopf algebra endomorphisms of $D$. For example, if $D$ is Hopf algebra that factorizes by the semidirect product of two Hopf algebras $A$, $H$ and $\omega_{D}\colon A\ltimes H\rightarrow D$ is the corresponding isomorphism of Hopf algebras, then
\[q=\omega_{D}\circ \bigl(\eta_{A}\otimes \bigl((\varepsilon_{A}\otimes H)\circ \omega_{D}^{-1}\bigr)\bigr)\colon D\rightarrow D\]
is an idempotent morphism of Hopf algebras. Remember that, in the particular case of groups (cocommutative Hopf algebras in {\sf Set}), it is well known that the set of idempotent endomorphisms~$q$ of a group $D$ are in one-to-one correspondence with the semidirect-product decompositions $A\ltimes H$ of $D$ where $A=\operatorname{Ker}(f)$, $H= q(D)$ and $\varphi_{A}\colon H\times A\rightarrow A$ is the adjoint action of $H$ on~$A$, i.e., $\varphi_A(h,a)=h a h^{-1}$. The operation on $A\ltimes H$ is defined by $(a,h)(b,l)=(a\varphi_{A}(h,b), hl)$.

On the other hand, if $D$ is cocommutative and $\Upsilon\colon D\rightarrow D$ is a $\phi$-twisted operator, then, by~\mbox{\cite[Proposition 3.6]{LW2}}, the triple $(D,\Upsilon,\phi\circ \Upsilon)$ is a Rota--Baxter system and, as a consequence, applying \cite[Proposition 3.8]{LW2} it is obtained that $\mathbb{D}_{\Upsilon}=\bigl(D,D_{\Upsilon},\sigma_{D}^{\Upsilon}\bigr)$ is a cocommutative Hopf truss, where
\begin{gather*}
	\sigma_{D}^{\Upsilon}\coloneqq\mu_{D}\circ (\Upsilon\otimes(\lambda_{D}\circ\phi\circ \Upsilon))\circ\delta_{D}=\Upsilon\ast(\lambda_{D}\circ\phi\circ \Upsilon)
\end{gather*}
and $H_{\Upsilon}$ is the non-unital bialgebra
\begin{gather*}
	H_{\Upsilon}=(H,\mu_{\Upsilon},\varepsilon_{D},\delta_{D}),
\end{gather*}
being
\begin{gather*}
	\mu_{\Upsilon}\coloneqq\mu_{D}\circ((\mu_{D}\circ(\Upsilon\otimes D))\otimes(\lambda_{D}\circ \phi\circ \Upsilon))\circ(D\otimes c_{D,D})\circ(\delta_{D}\otimes D),
\end{gather*}
i.e.,
\[
\includegraphics[scale=3]{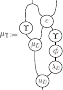}
\]

Thus, by Theorem \ref{funLambda}, the triple
\[
\Biggl({\rm id}_{D}\begin{array}{c}D\\\downarrow\\D_{\Upsilon}\end{array},\Gamma_{D}^{\sigma_{D}^{\Upsilon}},\sigma_{D}^{\Upsilon}\Biggr)
\]
is a weak twisted relative Rota--Baxter operator, where it is straightforward to compute that under these conditions
\[
\Gamma_{D}^{\sigma_{D}^{\Upsilon}}=\varphi_{D}^{{\rm ad}}\circ((\phi\circ \Upsilon)\otimes D).
\]
\end{Example}

The main results of this section are presented below. The first one is the following.

\begin{Theorem}\label{adjpair}
The functor $\Lambda$ is left adjoint to the functor $\Omega$.
\end{Theorem}
\begin{proof}
Let us construct a bijection
\begin{gather*}
{}^{\mathbb{H}}\Sigma_{T}\colon\ \Hom_{{\sf HTr}^{\star}}\Biggl(\mathbb{H},\Omega\Biggl(\Biggl(T\begin{array}{c}B\\\downarrow\\C\end{array},\varphi_{B},\Psi_{B}\Biggr)\Biggr)\Biggr)\longrightarrow \Hom_{{\sf wtr}\textnormal{-}{\sf RB}^{\star}}\Biggl(\Lambda(\mathbb{H}),\Biggl(T\begin{array}{c}B\\\downarrow\\C\end{array},\varphi_{B},\Psi_{B}\Biggr)\Biggr),
\end{gather*}
where
\[
\mathbb{H}=(H_{1},H_{2},\sigma_{H}) \qquad \text{and}\qquad \Biggl(T\begin{array}{c}B\\\downarrow\\C\end{array},\varphi_{B},\Psi_{B}\Biggr)
\]
 are arbitrary objects in ${\sf HTr}^{\star}$ and ${\sf wtr}\textnormal{-}{\sf RB}^{\star}$, respectively.

On the one hand, take
\[
f\colon \mathbb{H}\rightarrow \Omega\Biggl((T\begin{array}{c}B\\\downarrow\\C\end{array},\varphi_{B},\Psi_{B})\Biggr)=\widetilde{\mathbb{B}}=\bigl(B,\widetilde{B},\Psi_{B}\bigr)
\]
 a morphism in ${\sf HTr}^{\star}$. Let us see that $(f,T\circ f)$ is a morphism in ${\sf wtr}\textnormal{-}{\sf RB}^{\star}$ between the weak twisted relative Rota--Baxter operators \[
 \Lambda(\mathbb{H})=\Biggl({\rm id}_{H}\begin{array}{c}H_{1}\\\downarrow\\H_{2}\end{array},\Gamma_{H_{1}}^{\sigma_{H}},\sigma_{H}\Biggr)
 \qquad \text{and} \qquad \biggl(T\begin{array}{c}B\\\downarrow\\C\end{array},\varphi_{B},\Psi_{B}\Biggr).
 \]
 First of all note that $f\colon H_{1}\rightarrow B$ is a Hopf algebra morphism and $T\circ f\colon H_{2}\rightarrow C$ is a morphism of non-unital bialgebras because
\begin{align*}
&T\circ f\circ \mu_{H}^{2}\\
&\qquad{}=T\circ\widetilde{\mu}_{B}\circ (f\otimes f)\
\text{\footnotesize\textnormal{\big(by the condition of morphism of non-unital bialgebras for $f\colon H_{2}\rightarrow \widetilde{B}$\big)}}\\
&\qquad{}=\mu_{C}\circ((T\circ f)\otimes(T\circ f))\ \text{\footnotesize\textnormal{(by (i) of Definition \ref{WTRB})}}.
\end{align*}

Moreover, it is direct to compute that \eqref{mor0twrRb} holds in this case and \eqref{mor1twrRb} follows by \eqref{Cond mor truss}. Finally, we still have to check \eqref{mor2.0twrRb}. Indeed,
\begin{align*}
f\circ \Gamma_{H_{1}}^{\sigma_{H}}
={}&\mu_{B}\circ\bigl((f\circ\lambda_{H}\circ \sigma_{H})\otimes\bigl(f\circ\mu_{H}^{2}\bigr)\bigr)\circ(\delta_{H}\otimes H)\\
&{} \text{\footnotesize\textnormal{(by the condition of algebra morphism for $f\colon H_{1}\rightarrow B$)}}\\
={}&\mu_{B}\circ((\lambda_{B}\circ\Psi_{B}\circ f)\otimes (\widetilde{\mu}_{B}\circ(f\otimes f)))\circ(\delta_{H}\otimes H)\\
&{} \text{\footnotesize\textnormal{\big(by \eqref{morant}, \eqref{Cond mor truss} and the condition of morphism of non-unital bialgebras for $f\colon H_{2}\rightarrow \widetilde{B}$\big)}}\\
={}&\Gamma_{B}^{\Psi_{B}}\circ(f\otimes f)\\
&{} \text{\footnotesize\textnormal{\big(by the condition of coalgebra morphism for $f$ and definition of $\Gamma_{B}^{\Psi_{B}}$ for the Hopf truss $\widetilde{\mathbb{B}}$\big)}}\\
={}&\mathfrak{m}_{B}\circ(f\otimes f)\ \text{\footnotesize\textnormal{(by \eqref{GammaPsiH})}}\\
={}&\varphi_{B}\circ((T\circ f)\otimes f).
\end{align*}

So, we define
\[{}^{\mathbb{H}}\Sigma_{T}(f)=(f,T\circ f).\]

On the other hand, consider $(x,y)$ a morphism in ${\sf wtr}\textnormal{-}{\sf RB}^{\star}$ between the weak twisted relative Rota--Baxter
\[
\Lambda(\mathbb{H})=\Biggl({\rm id}_{H}\begin{array}{c}H_{1}\\\downarrow\\H_{2}\end{array},\Gamma_{H_{1}}^{\sigma_{H}},\sigma_{H}\biggr)
\qquad \text{and} \qquad \Biggl(T\begin{array}{c}B\\\downarrow\\C\end{array},\varphi_{B},\Psi_{B}\Biggr).
\]
 We will prove that $x$ is a morphism between the Hopf trusses $\mathbb{H}$ and $\widetilde{\mathbb{B}}$. To do this, it is sufficient to compute that $x\colon H_{2}\rightarrow \widetilde{B}$ is a morphism of non-unital bialgebras, what follows by
\begin{align*}
\widetilde{\mu}_{B}\circ(x\otimes x)
={}&\mu_{B}\circ((\Psi_{B}\circ x)\otimes(\varphi_{B}\circ((T\circ x)\otimes x)))\circ(\delta_{H}\otimes H)\\
&{}\text{\footnotesize\textnormal{(by the condition of coalgebra morphism for $x$)}}\\
={}&\mu_{B}\circ((x\circ \sigma_{H})\otimes(\varphi_{B}\circ(y\otimes x)))\circ(\delta_{H}\otimes H)\ \text{\footnotesize\textnormal{(by \eqref{mor0twrRb} and \eqref{mor1twrRb})}}\\
={}&\mu_{B}\circ(x\otimes x)\circ\bigl(\sigma_{H}\otimes\Gamma_{H_{1}}^{\sigma_{H}}\bigr)\circ(\delta_{H}\otimes H)\ \text{\footnotesize\textnormal{(by \eqref{mor2.0twrRb})}}\\
={}&x\circ\mu_{H}^{1}\circ\bigl(\sigma_{H}\otimes \Gamma_{H_{1}}^{\sigma_{H}}\bigr)\circ(\delta_{H}\otimes H)\\
&{}\text{\footnotesize\textnormal{(by the condition of algebra morphism for $x\colon H_{1}\rightarrow B$)}}\\
={}&x\circ\mu_{H}^{2}\ \text{\footnotesize\textnormal{(by \eqref{mu2Htruss})}}.
\end{align*}

Thus, let us define
\begin{gather*}
{}^{\mathbb{H}}\Theta_{T}\colon \Hom_{{\sf wtr}\textnormal{-}{\sf RB}^{\star}}\Biggl(\Lambda(\mathbb{H}),\Biggl(T\begin{array}{c}B\\\downarrow\\C\end{array},\varphi_{B},\Psi_{B}\Biggr)\Biggr)\longrightarrow \Hom_{{\sf HTr}^{\star}}\Biggl(\mathbb{H},\Omega\Biggl(\Biggl(T\begin{array}{c}B\\\downarrow\\C\end{array},\varphi_{B},\Psi_{B}\Biggr)\Biggr)\Biggr)
\end{gather*}
by ${}^{\mathbb{H}}\Theta_{T}((x,y))=x$.

To conclude we have to prove that ${}^{\mathbb{H}}\Theta_{T}$ is the inverse of ${}^{\mathbb{H}}\Sigma_{T}$. Indeed,
\begin{align*}
\bigl({}^{\mathbb{H}}\Theta_{T}\circ {}^{\mathbb{H}}\Sigma_{T}\bigr)(f)={}^{\mathbb{H}}\Theta_{T}((f, T\circ f))=f,
\end{align*}
and
\begin{align*}
\bigl({}^{\mathbb{H}}\Sigma_{T}\circ{}^{\mathbb{H}}\Theta_{T}\bigr)(x,y)={}^{\mathbb{H}}\Sigma_{T}(x)=(x,T\circ x)=(x,y)\ \text{\footnotesize\textnormal{(by \eqref{mor0twrRb} for the morphism $(x,y)$)}}.
\tag*{\qed}
\end{align*}
\renewcommand{\qed}{}
\end{proof}

Consider the full subcategory of ${\sf wtr}\textnormal{-}{\sf RB}^{\star}$ consisting of all weak twisted relative Rota--Baxter operators satisfying \eqref{classcocomfrakmH},
\[
\Biggl(T\begin{array}{c}H\\\downarrow\\B\end{array},\varphi_{H},\Psi_{H}\biggr),
\]
 such that $T$ is an isomorphism in ${\sf C}$. As a first consequence, when we work with this subcategory, note that, by (i) of Definition \ref{WTRB}, $\mu_{B}$ admits an expression in terms of $\widetilde{\mu}_{H}$ given by
\begin{equation*}%\label{muBTiso}
\mu_{B}=T\circ \widetilde{\mu}_{H}\circ\bigl(T^{-1}\otimes T^{-1}\bigr).
\end{equation*}

We will denote this subcategory by ${\sf wtr}\textnormal{-}{\sf RB}^{\star}_{{\sf iso}}$. Moreover, take into account that the image of the functor $\Lambda$ always lives in this subcategory because, for any $\mathbb{H}\in {\sf HTr}^{\star}$, ${\rm id}_{H}\colon H_{1}\rightarrow H_{2}$ is always an isomorphism in ${\sf C}$. As a result,
$\Lambda\colon \sf{HTr}^{\star}\longrightarrow {\sf wtr}\textnormal{-}{\sf RB}^{\star}_{{\sf iso}}.$

Thus, if we denote by $\Omega'$ the restriction of functor $\Omega$ to the subcategory ${\sf wtr}\textnormal{-}{\sf RB}^{\star}_{{\sf iso}}$, the following result states that $\Lambda$ and $\Omega'$ give rise to a categorical equivalence between ${\sf wtr}\textnormal{-}{\sf RB}^{\star}_{{\sf iso}}$ and~${\sf HTr}^{\star}$.
\begin{Theorem}\label{catequivHTr}
The categories $\sf{HTr}^{\star}$ and ${\sf wtr}$-${\sf RB}^{\star}_{{\sf iso}}$ are equivalent.
\end{Theorem}
\begin{proof}On the one hand, let $\mathbb{H}=(H_{1},H_{2},\sigma_{H})$ be an object in $\sf{HTr}^{\star}$. Then,
\[\Omega'\circ\Lambda={\sf id}_{\sf{HTr}^{\star}}.\]

Indeed,
\begin{align*}
\bigl(\Omega'\circ\Lambda\bigr)(\mathbb{H})=\Omega'\Biggl(\Biggl({\rm id}_{H}\begin{array}{c}H_{1}\\\downarrow\\H_{2}\end{array},\Gamma_{H_{1}}^{\sigma_{H}},\sigma_{H}\Biggr)\Biggr)=\bigl(H_{1},\widetilde{H}_{1},\sigma_{H}\bigr),
\end{align*}
where $\widetilde{\mu}_{H_{1}}=\mu_{H}^{2}$ by \eqref{mu2Htruss}. This implies that $\widetilde{H}_{1}=H_{2}$ and thus $\bigl(\Omega'\circ\Lambda\bigr)(\mathbb{H})=\mathbb{H}$.

On the other hand, we have that \[\Lambda\circ\Omega'\simeq{\sf id}_{{\sf wtr}\textnormal{-}{\sf RB}^{\star}_{{\sf iso}}}\]
because, if
\[
\Biggl(T\begin{array}{c}H\\\downarrow\\B\end{array},\varphi_{H},\Psi_{H}\Biggr)
\]
is an object in ${\sf wtr}$-${\sf RB}^{\star}_{{\sf iso}}$,
by \eqref{GammaPsiH},
\begin{align*}
(\Lambda\circ\Omega')\Biggl(\Biggl(T\begin{array}{c}H\\\downarrow\\B\end{array},\varphi_{H},\Psi_{H}\Biggr)\Biggr)
&{}=\Lambda\bigl(\bigl(H,\widetilde{H},\Psi_{H}\bigr)\bigr)\\
&{}=\Biggl({\rm id}_{H}\begin{array}{c}H\\\downarrow\\\widetilde{H}\end{array},\Gamma_{H}^{\Psi_{H}},\Psi_{H}\Biggr)
 =\Biggl({\rm id}_{H}\begin{array}{c}H\\\downarrow\\\widetilde{H}\end{array},\mathfrak{m}_{H},\Psi_{H}\Biggr),
\end{align*}
which is isomorphic to
\[
\Biggl(T\begin{array}{c}H\\\downarrow\\B\end{array},\varphi_{H},\Psi_{H}\Biggr)
\] in the category ${\sf wtr}$-${\sf RB}^{\star}_{{\sf iso}}$ via the isomorphism $({\rm id}_{H},T)$.
\end{proof}

If we denote by ${\sf coc}$-${\sf wtr}$-${\sf RB}_{{\sf iso}}$ to the full subcategory of ${\sf wtr}$-${\sf RB}$ whose objects are weak twisted relative Rota--Baxter operators,
\[
\Biggl(T\begin{array}{c}H\\\downarrow\\B\end{array},\varphi_{H},\Psi_{H}\Biggr),
\]
such that $T$ is an isomorphism and $H$ is cocommutative, then the following result is direct as a~consequence of the previous one.
\begin{Corollary}

The categories $\sf{HTr}^{\star}$, ${\sf wtr}$-${\sf RB}^{\star}_{{\sf iso}}$ and ${\sf wt}\textnormal{-}{\sf Post}\textnormal{-}{\sf Hopf}^{\star}$ are equivalent.
\end{Corollary}
\begin{proof}
The proof follows by Theorem \ref{th-iso-wtph-htr} and previous theorem.
\end{proof}

\begin{Corollary}
The categories ${\sf coc}\textnormal{-}{\sf HTr}$ and ${\sf coc}$-${\sf wtr}$-${\sf RB}_{{\sf iso}}$ are equivalent.
\end{Corollary}
\begin{proof}
It is enough to note that, under cocommutativity assumption, \eqref{classcocomGH} and \eqref{classcocomfrakmH} always hold.
\end{proof}

\begin{Corollary}
The categories ${\sf coc}\textnormal{-}{\sf HTr}$, ${\sf coc}$-${\sf wtr}$-${\sf RB}_{{\sf iso}}$ and $\sf{coc}\textnormal{-}{\sf wt}\textnormal{-}{\sf Post}$-$\sf{Hopf}$ are equivalent.
\end{Corollary}
\begin{proof}
The proof follows by Corollary \ref{cor-iso-cocwtph-cochtr} and previous corollary.
\end{proof}

\subsection*{Acknowledgements}

The authors were supported by Ministerio de Ciencia e Innovaci\'on of Spain. Agencia Estatal de Investigaci\'on. Uni\'on Europea -- Fondo Europeo de Desarrollo Regional (FEDER). Grant PID2020-115155GB-I00: Homolog\'{\i}a, homotop\'{\i}a e invariantes categ\'oricos en grupos y \'algebras no asociativas. Moreover, Jos\'e Manuel Fern\'andez Vilaboa and Brais Ramos P\'erez were funded by Xunta de Galicia, grant ED431C 2023/31 (European FEDER support included, UE). Also, Brais Ramos P\'erez was financially supported by Xunta de Galicia Scholarship ED481A-2023-023.
The authors are grateful to the referees for their thoughtful comments, which contributed to the improvement of the manuscript.

\pdfbookmark[1]{References}{ref}
\LastPageEnding

\end{document}